\documentclass{siamart250211}
\usepackage{booktabs}
\usepackage{multirow}

\usepackage{lipsum}
\usepackage{amsfonts}
\usepackage{graphicx}
\usepackage{epstopdf}
\usepackage{algorithmic}
\ifpdf
  \DeclareGraphicsExtensions{.eps,.pdf,.png,.jpg}
\else
  \DeclareGraphicsExtensions{.eps}
\fi

\newsiamremark{remark}{Remark}
\newsiamremark{hypothesis}{Hypothesis}
\crefname{hypothesis}{Hypothesis}{Hypotheses}
\newsiamthm{claim}{Claim}
\newsiamremark{fact}{Fact}
\crefname{fact}{Fact}{Facts}

\headers{Parametric Probabilistic Manifold Decomposition}{J. Guo, and D. Xiao}

\title{Parametric Probabilistic Manifold Decomposition for Nonlinear Model Reduction\thanks{Submitted to the editors DATE.
\funding{The authors acknowledge the support of the Fundamental Research Funds for the Central Universities, the Top Discipline Plan of Shanghai Universities-Class I and Shanghai Municipal Science and Technology Major Project (No. 2021SHZDZX0100), National Key $R\&D$ Program of China(NO.2022YFE0208000, 2024YFC2816400 and 2024YFC2816401).}}}

\author{Jiaming Guo\thanks{Shanghai Research Institute for Intelligent Autonomous Systems, Tongji University, Shanghai 201210, CHINA. 
 (\email{2411955@tongji.edu.cn}).}
\and Dunhui Xiao\thanks{School of Mathematical Sciences,
Key Laboratory of Intelligent Computing and Applications (Ministry of Education), Tongji University, Shanghai 200092, CHINA. 
  (\email{xiaodunhui@tongji.edu.cn}).}
}

\usepackage{amsopn}

\ifpdf
\hypersetup{
  pdftitle={Parametric Probabilistic Manifold Decomposition for Nonlinear Model Reduction},
  pdfauthor={Jiaming Guo, and Dunhui Xiao}
}
\fi

\begin{document}

\maketitle

\begin{abstract} 
Probabilistic Manifold Decomposition (PMD)\cite{doi:10.1137/25M1738863}, developed in our earlier work, provides a  nonlinear model reduction by embedding high-dimensional dynamics onto low-dimensional probabilistic manifolds. The PMD has demonstrated strong performance for time-dependent systems. However, its formulation is for temporal dynamics and does not directly accommodate parametric variability, which limits its applicability to tasks such as design optimization, control, and uncertainty quantification.
In order to address the limitations, a \emph{Parametric Probabilistic Manifold Decomposition} (PPMD) is presented to deal with parametric problems. The central advantage of PPMD is its ability to construct continuous, high-fidelity parametric surrogates while retaining the transparency and non-intrusive workflow of PMD. By integrating probabilistic-manifold embeddings with parameter-aware latent learning, PPMD enables smooth predictions across unseen parameter values (such as different boundary or initial conditions). To validate the proposed method, a comprehensive convergence analysis is established for PPMD, covering the approximation of the linear principal subspace, the geometric recovery of the nonlinear solution manifold, and the statistical consistency of the kernel ridge regression used for latent learning. The framework is then numerically demonstrated on two classic flow configurations: flow past a cylinder and backward-facing step flow. Results confirm that PPMD achieves superior accuracy and generalization beyond the training parameter range compared to the conventional proper orthogonal decomposition with Gaussian process regression (POD+GPR) method.


\end{abstract}



\section{Introduction}
In computational engineering, the high-fidelity simulations of physical systems are inherently expensive because the discretization of the governing partial differential equations yields extremely large state spaces with millions or billions of degrees of freedom (DoF) \cite{diez2021nonlinear}. These high-dimensional models become prohibitively costly in many-query scenarios such as optimization, uncertainty quantification, or real-time control, where solutions must be computed for numerous parameter configurations.

Fortunately, the intrinsic structure of these systems often enables a significant reduction in computational complexity. In practice, the solution set typically resides on a low-dimensional manifold embedded in the high-dimensional state space, governed by a relatively small number of latent variables. This observation provides the theoretical foundation for reduced-order modeling (ROM)—a methodology that projects the governing equations onto a much smaller subspace while preserving the essential physical dynamics.

Extending ROM to parametric problems introduces the need for parametric ROM (pROM), which captures not only the system’s dynamics but also its dependence on physical, geometric, or operational parameters. A well-constructed pROM provides a computationally efficient surrogate capable of approximating the full-order model across the parameter space. This enables rapid exploration of design spaces, parameter estimation, and robust optimization—all with dramatically reduced computational overhead.

Thus, by exploiting the low-dimensional structure embedded in high-fidelity data, parametric ROMs bridge the gap between accuracy and efficiency, making advanced computational tasks feasible in engineering practice\cite{benner2015survey,baur2011interpolatory}.

Classical projection-based ROMs (proper orthogonal decomposition (POD), balanced truncation, reduced-basis, and rational or interpolatory methods) provide efficient linear-subspace approximations and admit rigorous error measures under favorable conditions \cite{berkooz1993proper,heinkenschloss2008balanced,peterson1989reduced,baur2011interpolatory,benner2015survey, mlinaric2023unifying}. However, restricting the state to evolve within a linear subspace imposes an intrinsic limitation for strongly nonlinear problems \cite{mcquarrie2023nonintrusive,schlegel2015long,osth2014need,chaturantabut2010nonlinear}. In parametric applications, this limitation manifests as poor generalization across different parameter values: a global linear basis requires a larger number of modes to achieve acceptable accuracy, which increases computational savings and complicates the offline and online decomposition \cite{benner2015survey,bui2008model,mlinaric2023optimal}.

To mitigate the cost of nonlinear operators while retaining projection structure, operator-level accelerations such as the empirical interpolation method (EIM), discrete empirical interpolation (DEIM), residual-DEIM, and Gauss–Newton with approximated tensors (GNAT) have been developed \cite{barrault2004empirical,peherstorfer2014localized,xiao2014non,carlberg2013gnat}. These methods reduce the expense of evaluating nonlinear terms and enable affine-like online complexity; nevertheless, they do not eliminate the fundamental geometric limitation of linear-subspace ROMs. In particular, interpolating or approximating nonlinear operators addresses computational bottlenecks but does not guarantee that the reduced coordinates form a faithful, low-dimensional parametrization of a curved solution manifold across the parameter domain, especially for nonaffine parameter dependence or when parametric effects interact nonlinearly with spatial dynamics \cite{benner2015survey,baur2011interpolatory}.

An alternative route is nonlinear manifold learning (Isomap, LLE, diffusion maps, autoencoders, and related techniques), which seeks nonlinear coordinate charts that are better aligned with the intrinsic geometry of the solution set \cite{balasubramanian2002isomap,roweis2000nonlinear,coifman2006diffusion,lee2020model,zhang2004principal}. These techniques can yield more compact representations for strongly nonlinear problems and have been combined with regression or operator-learning methods to produce parametric surrogates \cite{xie2018data}. Yet, the practical deployment of nonlinear embeddings in scalable pROM pipelines faces multiple obstacles. First, preserving global dynamical consistency and fine-scale local structure across the entire parameter domain remains challenging, especially when latent representations are learned purely from data without explicit connection to the underlying equations\cite{zhang2015efficient}. Second, many nonlinear embedding schemes (e.g. autoencoders) are sensitive to hyperparameter choices, sampling density, and network architecture, which can lead to instability, poor extrapolation outside the training set, or latent degeneracies \cite{fu2023non}. Such sensitivity often requires problem-specific tuning and extensive offline training, undermining robustness in many-query parametric settings \cite{benner2015survey,bui2008model}. Third, constructing reliable offline/online splits that support rapid online evaluation consistently across parameter values is difficult when nonlinear manifolds lack affine or other exploitable structures, and adaptive or hyper-reduction strategies may be required to maintain efficiency \cite{pagliantini2023gradient}. Moreover, the computational cost of training and deploying large nonlinear models can be high for large datasets or high-dimensional parameter spaces, even when compared to projection-based alternatives \cite{baur2011interpolatory}. In particular, recent studies on nonlinear reduced-order modeling highlight that deep, convolutional, or manifold-based embeddings can incur significant offline costs and may still require physics-aware regularization to maintain interpretability and extrapolation quality \cite{higham2019deep,solomon1979connection,absil2012projection}, which collectively inhibit the straightforward replacement of projection-based pROMs with black-box nonlinear embeddings in many engineering applications.

Probabilistic Manifold Decomposition (PMD) \cite{doi:10.1137/25M1738863} provides a principled, interpretable nonlinear model reduction method that combines linear reduction with a probabilistic manifold representation of nonlinear residuals built from Markov processes and geodesic-aware similarity graphs. However, it is inherently formulated for time-dependent problems and is not suitable for parameter-dependent settings.

Thus, \emph{Parametric Probabilistic Manifold Decomposition} (PPMD) is proposed, which extends PMD to parameterized dynamical systems. PPMD assembles a unified snapshot matrix by concatenating time-resolved trajectories across parameter samples, then decomposes the data into a dominant linear subspace (energetic modes) and a nonlinear residual embedded on a low-dimensional probabilistic manifold. For continuous, accurate prediction across parameters, PPMD uses recursive, adaptive kernel ridge regression to synthesize additional latent samples and improve local parametric coverage, converting these latent sequences into smooth parameters and obtaining latent maps via weighted smoothing-spline interpolation for high-fidelity reconstruction. The resulting method retains PMD’s transparency and operator structure while yielding uncertainty-aware, geometry-respecting parametric surrogates with explicit, interpretable error sources (linear truncation, manifold embedding, regression, and smoothing).

The performance of PPMD is demonstrated on two classical benchmark fluid problems: unsteady flow past a cylinder and backward-facing step flow. Comparisons with a traditional POD combined with Gaussian Process Regression (GPR) ROM illustrate that PPMD delivers significantly improved accuracy, smoother parametric generalization, and higher robustness. 

The remainder of the paper is organized as follows. Section~\ref{sec:govern} presents the governing equations. Section~\ref{sec:review_pmd} provides a brief description of the PMD method. Section~\ref{sec:method} details the PPMD methodology, including linear reduction, probabilistic manifold construction, recursive latent learning, and continuous parameter mapping. Section~\ref{sec:ppmd_convergence} develops the theoretical analysis of PPMD, providing convergence guaranties and error decomposition. Section~\ref{sec:experiments} demonstrates numerical experiments for parametric CFD problems. Finally, conclusions and future work are provided in Section~\ref{sec:summary}.

\section{Governing equations} \label{sec:govern}

The physical system considered in this work is the three dimensional (3D) non-hydrostatic and incompressible Navier-Stokes equations, which describe the conservation laws of fluid motion. The mass conservation (continuity) equation is given by:
\begin{equation}\label{eq:continuity}
\nabla \cdot \mathbf{u} = 0,
\end{equation}
which enforces the incompressibility condition, ensuring that the divergence of the velocity field $\mathbf{u}$ vanishes at every point in space. This constraint implies the conservation of mass in the absence of sources or sinks. This condition implies that the fluid is incompressible, with no local volume change in the flow field.

The momentum balance is described by:
\begin{equation}\label{eq:equation2}
\frac{\partial \mathbf{u}}{\partial t} + (\mathbf{u} \cdot \nabla) \mathbf{u} + f , \mathbf{k} \times \mathbf{u} = -\nabla p + \nabla \cdot \boldsymbol{\tau},
\end{equation}
where the left-hand side represents the temporal and convective acceleration of the fluid, along with Coriolis effects due to system rotation, where $f$ is the Coriolis parameter and $\mathbf{k}$ is the unit vector aligned with the rotation axis.

On the right-hand side, $p := p / \rho_0$ is the rescaled pressure term, and $\boldsymbol{\tau}$ denotes the viscous stress tensor, defined as:
\begin{equation}
\boldsymbol{\tau} := \mu \left( \nabla \mathbf{u} + (\nabla \mathbf{u})^T \right),
\end{equation}
with $\mu$ being the dynamic viscosity of the fluid. This term accounts for internal viscous forces due to velocity gradients.

These equations form the high-fidelity governing model used for simulation. The proposed PPMD framework seeks to efficiently model and predict such parameter-dependent dynamics in a reduced-order setting.

\section{Probabilistic Manifold Decomposition (PMD)}
\label{sec:review_pmd}

PMD is a recently presented nonlinear model reduction method \cite{doi:10.1137/25M1738863}. PMD separates dominant linear dynamics from nonlinear residuals, embeds the latter on a low-dimensional manifold via Markov diffusion geometry, and applies reduced-order predictors in both latent spaces. This section provides a concise summary of the PMD workflow, emphasizing the mathematical structure relevant to its parametric extension.

\subsection{Linear feature extraction and residual separation}
Given a snapshot matrix $\mathcal{U}\in\mathbb{R}^{n\times m}$, standardized as
\begin{equation}
\bar{\mathcal{U}}=\left[\frac{\mathbf{u}_1-\mu}{\sigma},\dots,\frac{\mathbf{u}_m-\mu}{\sigma}\right],
\end{equation}
PMD applies a truncated singular value decomposition (SVD)
\begin{equation}
\bar{\mathcal{U}}=\Psi\Lambda\Theta^\top,
\end{equation}
retaining the smallest $r$, such that
\begin{equation}
\sum_{i=1}^r \lambda_i^2 \ge (1-\varepsilon)\sum_{j=1}^m \lambda_j^2.
\end{equation}
The linear latent coordinates and corresponding approximation are
\begin{equation}
\mathcal{Z}=\Lambda_r\Theta_r^\top,\qquad 
\bar{\mathcal{U}}_r=\Psi_r\Lambda_r\Theta_r^\top,
\end{equation}
and the nonlinear residuals are
\begin{equation}
\Xi
=(I-\Theta_r\Theta_r^\top)\bar{\mathcal{U}}
\in\mathbb{R}^{n\times m}.
\end{equation}

\subsection{Probabilistic manifold construction}
Given the residual samples $\Xi=[\mathbf{r}_1,\dots,\mathbf{r}_m]$, PMD constructs a weighted graph with Gaussian kernel
\begin{equation}
\mathcal{A}(i,j)
=\exp\!\left(-\frac{\|\mathbf{r}_i-\mathbf{r}_j\|^2}{\varepsilon^2}\right).
\end{equation}
Geodesic distances $\rho_g(\mathbf{r}_i,\mathbf{r}_j)$ are computed using Floyd–Warshall updates:
\begin{equation}
\rho_g^{(k)}(i,j)
=\min\big(\rho_g^{(k-1)}(i,j),\;\rho_g^{(k-1)}(i,k)+\rho_g^{(k-1)}(k,j)\big).
\end{equation}

A Markov transition matrix is defined by row-normalization,
\begin{equation}
\mathcal{P}(i,j)=\frac{\mathcal{A}(i,j)}{\sum_k \mathcal{A}(i,k)},
\qquad 
\tilde{\mathcal{P}}=\mathcal{P}^t.
\end{equation}
Spectral decomposition
\begin{equation}\label{basis}
\tilde{\mathcal{P}}\varphi_k=\lambda_k\varphi_k,
\end{equation}
produces the nonlinear embedding
\begin{equation}
\Phi(\mathbf{r}_i)
=
\begin{bmatrix}
\lambda_1^t \varphi_1(i)\\[-2pt]
\vdots\\[-2pt]
\lambda_r^t \varphi_r(i)
\end{bmatrix},
\qquad 
\Phi=[\Phi(\mathbf{r}_1),\dots,\Phi(\mathbf{r}_m)]\in\mathbb{R}^{r\times m}.
\end{equation}

\subsection{Dynamic prediction in latent spaces}

After dimensionality reduction, PMD models the temporal evolution of the system separately for the linear and nonlinear manifolds.

\subsubsection{Linear latent dynamics}  
Given $Z=\mathcal{Z}$, PMD fits a regularized linear operator
\begin{equation}
A_{\mathrm{lin}}
=Z^{(2)} Z^{(1)\top}
\left(Z^{(1)} Z^{(1)\top}+\lambda I\right)^{-1},
\end{equation}
with $Z^{(1)}=[z_1,\dots,z_{m-1}]$, $Z^{(2)}=[z_2,\dots,z_m]$, and predicts recursively
\begin{equation}
z_{k}=A_{\mathrm{lin}}^k z_0.
\end{equation}

This prediction step is fully explicit, computationally inexpensive, and preserves the dominant linear structure learned by the SVD.

\subsubsection{Nonlinear dynamics on the probabilistic manifold}
After predicting the linear components, PMD models the evolution of nonlinear residuals on the probabilistic manifold. Let $\Phi = [\boldsymbol{\phi}_1, \dots, \boldsymbol{\phi}_m] \in \mathbb{R}^{r \times m}$ denote the low-dimensional embeddings, with each $\boldsymbol{\phi}_i = \Phi(\mathbf{r}_i)$. A Gaussian kernel based on geodesic distance $\rho_g$ defines similarity: 
\begin{equation}
K(\boldsymbol{\phi}_i, \boldsymbol{\phi}_j) = \exp\left( -\frac{\rho_g(\boldsymbol{\phi}_i, \boldsymbol{\phi}_i)^2}{\varepsilon^2} \right),
\end{equation}
normalized to a Markov matrix:
\begin{equation}
\Pi = \mathcal{D}^{-1} \mathcal{K}, \quad \mathcal{D}_{ii} = \sum_{j=1}^m \mathcal{K}(\boldsymbol{\phi}_i, \boldsymbol{\phi}_j).
\end{equation}
Eigen-decomposition $\Pi \mathbf{u}_j = \zeta_j \mathbf{u}_j$ yields harmonics $1 = \zeta_1 \ge \dots \ge \zeta_m \ge 0$. Temporal evolution is approximated via ridge regression on shifted embeddings
\begin{equation}
\Phi_1 = [\boldsymbol{\phi}_1, \dots, \boldsymbol{\phi}_{m-1}], \quad
\Phi_2 = [\boldsymbol{\phi}_2, \dots, \boldsymbol{\phi}_m], \quad
\Omega = \Phi_2 \Phi_1^\top (\Phi_1 \Phi_1^\top + \lambda I)^{-1}.
\end{equation}
Rows $\boldsymbol{\omega}_i$ of $\Omega$ are projected onto $\{\mathbf{u}_j\}$:
\begin{equation}
d_j = \begin{bmatrix} \langle \boldsymbol{\omega}_1, \mathbf{u}_j \rangle \\ \vdots \\ \langle \boldsymbol{\omega}_r, \mathbf{u}_j \rangle \end{bmatrix} \in \mathbb{R}^r.
\end{equation}
For a new point $\boldsymbol{\phi}_{\text{new}}$, the harmonics are interpolated:
\begin{equation}
U_j(\boldsymbol{\phi}_{\text{new}}) = \frac{1}{\zeta_j} \sum_{i=1}^m \mathcal{K}(\boldsymbol{\phi}_{\text{new}}, \boldsymbol{\phi}_i)\, \mathbf{u}_j[i],
\end{equation}
yielding the nonlinear prediction:
\begin{equation}
\boldsymbol{\phi}_{\text{new}} = \sum_{j=1}^r d_j U_j(\boldsymbol{\phi}_{\text{new}}).
\end{equation}
This approach combines spectral harmonics with manifold-aware interpolation for accurate nonlinear dynamics prediction.

\subsection{Reconstruction via lifting}
Given the predicted latent coordinates $z(t)$ and $\phi(t)$, PMD reconstructs the high-dimensional state using an additive nonlinear lifting:
\begin{equation}
\hat{u}(t)
=\Psi_r z(t) + \mathcal{K}\phi(t),
\end{equation}
where $\Psi_r$ contains the retained SVD modes and $\mathcal{K}$ reconstructs nonlinear residuals. The lifting map $\mathcal{K}$ is obtained by
\begin{equation}
\mathcal{K}=\arg\min_{\mathcal{K}}
\|\mathcal{K}\Phi-\Xi\|_F^2+\lambda\|\mathcal{K}\|_F^2.
\end{equation}

PMD thus yields a fully explicit, interpretable surrogate that couples linear SVD modes with a probabilistic nonlinear manifold, enabling efficient reduced-order prediction for strongly nonlinear dynamical systems.

\section{Parametric probabilistic manifold decomposition(PPMD)}\label{sec:method}


In this section, an extension of PMD is proposed to parameter-dependent systems, referred to as \textit{Parametric PMD} (PPMD). The objective of PPMD is to learn a reduced-order model that generalizes across a continuous range of parameter values, enabling efficient prediction for unseen parameter configurations. Several critical modifications are introduced to PPMD in the prediction stage to accommodate parametric variation. A kernel ridge regression framework with dynamically adjusted polynomial kernels and recursive augmentation strategies is adopted to enhance generalization and stability.

Importantly, PPMD overcomes the discretization constraints of classical PMD by incorporating a continuous-parameter prediction mechanism based on cubic smoothing splines, which provides smooth reconstructions across the entire parameter domain and enables evaluation at arbitrary parameter values within the training range. This approach not only supports multi-parameter generalization but also ensures smooth and accurate predictions throughout the parameter space.

\subsection{Snapshot Matrix Construction}

Unlike traditional PMD, which constructs the snapshot matrix using time-discretized fields under a fixed parameter setting, the PPMD framework constructs its snapshot matrix by varying the physical or design parameters while aggregating full time-resolved trajectories.

Let \( \{ \mu_1, \dots, \mu_{n_s} \} \subset \mathbb{R}^p \) denote a collection of \( n_s \) parameter samples. For each parameter \( \mu_i \), suppose the high-fidelity simulation yields a trajectory of state vectors \( \mathbf{u}^{(i)}(t_j) \in \mathbb{R}^{n} \) at \( m \) time steps \( \{ t_1, \dots, t_m \} \), where \( n \) is the number of spatial degrees of freedom. The time series for each parameter are concatenated into a single vector:
\begin{equation}
\mathbf{u}(\mu_i) = 
\begin{bmatrix}
\mathbf{u}^{(i)}(t_1) \\
\mathbf{u}^{(i)}(t_2) \\
\vdots \\
\mathbf{u}^{(i)}(t_m)
\end{bmatrix}
\in \mathbb{R}^{N}, \quad \text{where } N = n \cdot m.
\end{equation}

The full parametric snapshot matrix is then constructed by concatenating these trajectories across different parameters:
\begin{equation}
U = 
\begin{bmatrix}
\mathbf{u}(\mu_1) & \mathbf{u}(\mu_2) & \cdots & \mathbf{u}(\mu_{n_s})
\end{bmatrix}
\in \mathbb{R}^{N \times n_s}.
\end{equation}  				

Each column of \( U \) corresponds to a complete spatiotemporal trajectory at a particular parameter value \( \mu_i \). This construction enables PPMD to simultaneously exploit temporal coherence and parametric variation in the system response.

\subsection{Dimensionality Reduction and Manifold Embedding} \label{PPMDnonlinear}

A truncated singular value decomposition is applied to the normalized snapshot matrix:
\begin{equation}
\bar{U}_p = \Psi_p \Lambda_p \Theta_p^T,
\end{equation}
where \( \bar{U}_p \in \mathbb{R}^{N \times n_s} \) denotes the column-wise normalized snapshot matrix obtained from \( U \), and \( \Psi_p \in \mathbb{R}^{N \times r^p} \), \( \Lambda_p \in \mathbb{R}^{r^p \times r^p} \), and \( \Theta_p \in \mathbb{R}^{n_s \times r^p} \) contain the leading left singular vectors, singular values, and right singular vectors, respectively.

The corresponding rank-\( r^p \) linear approximation of the normalized snapshot matrix is given by:
\begin{equation}\label{parametric linear}
\bar{U}_p^{(r^p)} = \Psi_p^{(r^p)} \Lambda_p^{(r^p)} {\Theta_p^{(r^p)}}^T \in \mathbb{R}^{N \times n_s},
\end{equation}
where $\Psi_p^{(r^p)}\in\mathbb{R}^{N \times r}$, $\Lambda_p^{(r^p)}\in\mathbb{R}^{r^p \times r^p}$, and $\Theta_p^{(r^p)}\in\mathbb{R}^{n_s \times r^p}$ are truncated matrices containing the first $r^p$ components.

The dominant linear features are retained in the low-dimensional coordinate matrix:
\begin{equation}
Z_p =  \Lambda_p^{(r^p)} {\Theta_p^{(r^p)}}^T \in \mathbb{R}^{r^p \times n_s}.
\end{equation}

The residual matrix that encodes nonlinear components is computed as:
\begin{equation} \label{residual}
R_p = \bar{U}_p - \bar{U}_p^{(r^p)} = \bar{U}_p - \Theta_p^{(r^p)} {\Theta_p^{(r^p)}}^T \bar{U}_p \in \mathbb{R}^{N \times n_s},
\end{equation}
which projects the data onto the orthogonal complement of the linear subspace.

To extract nonlinear features, the residual matrix \( R_p \) is embedded into a probabilistic manifold as Section~\ref{sec:review_pmd} shown. A weighted graph is first built by defining a Gaussian affinity matrix:
\begin{equation}
\mathcal{A}_p(i, j) = \exp\left( -\frac{\|\mathbf{p}_i - \mathbf{p}_j\|^2}{\varepsilon_p^2} \right), \quad i,j = 1, \dots, n_s,
\end{equation}
where \( \varepsilon_p > 0 \) is a kernel bandwidth. To better reflect manifold geometry, geodesic distances \( \|\mathbf{p}_i - \mathbf{p}_j\| \) are approximated using Floyd–Warshall's algorithm over the \( k_p \)-nearest neighbor graph.

The resulting affinity matrix is normalized row-wise to construct one-step Markov transition matrix:
\begin{equation}
\mathcal{P}_p(i, j) = \frac{\mathcal{A}_p(i, j)}{\sum_{k=1}^{n_s} \mathcal{A}_p(i, k)}.
\end{equation}
To incorporate multi-step transfer behavior, the matrix is raised to the \( t \)-th power:
\begin{equation}
\tilde{\mathcal{P}}_p = \mathcal{P}_p^t.
\end{equation}

A spectral embedding is then obtained by computing the leading eigenpairs:
\begin{equation}
\tilde{\mathcal{P}}_p \varphi^{(p)}_k = \lambda^{(p)}_k \varphi^{(p)}_k, \quad k = 1, \dots, r^p,
\end{equation}
where \( \lambda^{(p)}_1 \geq \lambda^{(p)}_2 \geq \dots \) and \( \varphi^{(p)}_k \in \mathbb{R}^{n_s} \) denote the right eigenvectors. The first trivial eigenvector (constant) is discarded. The embedding of parameter sample \( \mu_i \) is given by:
\begin{equation}
\boldsymbol{\phi}^{(p)}_i =
\begin{bmatrix}
{\lambda_1^{(p)}}^t \varphi^{(p)}_1(i) \\
{\lambda_2^{(p)}}^t \varphi^{(p)}_2(i) \\
\vdots \\
{\lambda_{r^p}^{(p)}}^t \varphi_{r^p}^{(p)}(i)
\end{bmatrix} \in \mathbb{R}^{r^p}.
\end{equation}

Stacking all \( n_s \) embeddings yields the nonlinear coordinate matrix:
\begin{equation}\label{reductionresidual}
\Phi_p =
\begin{bmatrix}
\boldsymbol{\phi}_1 & \boldsymbol{\phi}_2 & \cdots & \boldsymbol{\phi}_{n_s}
\end{bmatrix}
\in \mathbb{R}^{r^p \times n_s}.
\end{equation}

The resulting low-dimensional representation \( \Phi_p \in \mathbb{R}^{r^p \times n_s} \) captures the intrinsic nonlinear structure of the parameter-dependent residuals and enables efficient learning of parametric mappings for downstream prediction and reconstruction tasks in PPMD.

\subsection{Parametric Prediction in Latent Space} \label{sec:parametric_prediction}

In contrast to the original PMD framework, which focuses on time-stepping predictions over fixed temporal snapshots, PPMD aims to generalize across varying physical parameters. Therefore, instead of predicting the system evolution over time, the core task in PPMD is to predict latent representations corresponding to unseen parameter values \( \mu^* \in \mathbb{R}^p \). This change in predictive objective necessitates a different modeling strategy that accounts for nonlocal, nonlinear dependence on parameters.

To enable parametric prediction, a unified framework based on dynamic kernel ridge regression (KRR) is proposed, which learns mappings from parameter values to both linear and nonlinear latent coordinates—\( Z_p \in \mathbb{R}^{r^p} \) and \( \Phi_p \in \mathbb{R}^{r^p} \), respectively. Departing from the static single step KRR used in PMD, PPMD introduces a recursive prediction strategy with dynamic kernel updates, enabling accurate predictions at new parameter points and exhibiting strong generalization even for unseen parametric configurations.

For the latent matrix \( Z_p \in \mathbb{R}^{r^p \times n_s} \), which encodes the linear features across parameter samples, two sequential partitions are defined for recursive training:
\begin{equation}
Z_p^{(1:n_s-1)} = [\mathbf{z}_1, \dots, \mathbf{z}_{n_s-1}] \in \mathbb{R}^{r^p \times (n_s - 1)}, \quad
Z_p^{(2:n_s)} = [\mathbf{z}_2, \dots, \mathbf{z}_{n_s}] \in \mathbb{R}^{r^p \times (n_s - 1)}.
\end{equation}

A kernel ridge regression model is trained to learn the transition relationship from one parameter state to the next.
\begin{equation}
\mathcal{K}_p = \arg\min_{\mathcal{K}_p \in \mathbb{R}^{(r^p) \times (n_s - 1)}} 
\left\| \mathcal{K}_p \, \kappa(Z_p^{(1:n_s-1)}) - Z_p^{(2:n_s)} \right\|_F^2 
+ \lambda_p \| \mathcal{K}_p \|_F^2,
\end{equation}
where the kernel matrix is computed using the polynomial kernel:
\begin{equation}
\kappa(Z_p^{(1:n_s-1)}) = ({Z_p^{(1:n_s-1)}}^\top Z_p^{(1:n_s-1)} + c_p)^{d_p},
\end{equation}
with tunable hyperparameters: polynomial degree \( d_p \), offset \( c_p \), and regularization coefficient \( \lambda_p \).

Using the learned mapping \( \mathcal{K}_p \), the predicted latent vector at the next parameter value \( \mu_{n_s+1} \) is obtained as:
\begin{equation}
\hat{\mathbf{z}}_{n_s+1} = \mathcal{K}_p \, \kappa(Z_p^{(2:n_s)}).
\end{equation}

The predicted vector is appended to form extended training data:
\begin{equation}
Z_p^{(1:n_s)} \leftarrow [Z_p^{(1:n_s-1)}, \mathbf{z}_{n_s}], \quad
Z_p^{(2:n_s+1)} \leftarrow [Z_p^{(2:n_s)}, \hat{\mathbf{z}}_{n_s+1}].
\end{equation}

To perform multi-step forecasting, the procedure is applied recursively. At each step, the newly predicted latent vector \( \hat{\mathbf{z}}_{n_s+k} \) (for \( k \geq 1 \)) is added to the training matrices:
\begin{equation}
Z_p^{(1:n_s+k)} \leftarrow [Z_p^{(1:n_s+k-1)}, \hat{\mathbf{z}}_{n_s+k-1}], \quad
Z_p^{(2:n_s+k+1)} \leftarrow [Z_p^{(2:n_s+k)}, \hat{\mathbf{z}}_{n_s+k}],
\end{equation}
followed by retraining the kernel ridge regression model with the updated kernel matrix and latent pairs.

This recursive dynamic KRR procedure generates a sequence of predicted latent vectors that extends the original discrete parameter samples into a denser latent trajectory, thereby preparing data suitable for continuous interpolation. Kernel hyperparameters \((d_p,c_p)\) (and the regularizer \(\lambda_p\)) are updated periodically using cross-validation or reconstruction-error criteria to adapt to local variations in parameter space. The same algorithm is applied independently to the nonlinear latent matrix \(\Phi_p\in\mathbb{R}^{r^p\times n_s}\) with its own estimator \(\mathcal{K}_\Phi\) and kernel parameters \((d_\phi,c_\phi,\lambda_\phi)\). 

Compared to single-step static regression, this recursive adaptive KRR generates extended latent sequences for better parameter coverage and robustness to local nonlinearities. Unlike classical PMD limited to pointwise extrapolation, this strategy enables PPMD to achieve improved generalization and accuracy across the entire parameter interval.

\subsection{From discrete latent samples to a dense continuous parameter trajectory}
\label{ppmd_continuous_interpolation}

The recursive dynamic KRR procedure in Section~\ref{sec:parametric_prediction} produces a sequence of latent vectors for a discrete set of parameter samples.
To obtain a continuous representation over the parameter domain and enable predictions at arbitrary parameter values \(\mu^\ast\), these discrete latent samples are further mapped onto a dense parameter interval.
This mapping uses smoothing splines in the latent space, which regularize the discrete trajectories, provide a differentiable parameter-to-state relation, and facilitate smooth parametric field reconstruction for analysis.

Let \(\{\mu_i\}_{i=1}^{N_p}\) denote the ordered set of parameter samples (original plus recursively predicted ones) and, as before, let
\begin{equation}
Z_p = \bigl[z_1,\dots,z_{N_p}\bigr]\in\mathbb{R}^{r^p\times N_p},\qquad
\Phi_p = \bigl[\phi_1,\dots,\phi_{N_p}\bigr]\in\mathbb{R}^{r^p\times N_p}
\end{equation}
be the matrices of linear and nonlinear latent coordinates, respectively (each column corresponds to one parameter sample). Our goal is to construct smooth functions
\begin{equation}
s_{z,j}(\mu)\quad (j=1,\dots,r^p),\qquad s_{\phi,j}(\mu)\quad (j=1,\dots,r^p)
\end{equation}
that interpolates or smooths the discrete coordinates \(Z_p(j,i)\) and \(\Phi_p(j,i)\) as functions of the scalar parameter \(\mu\).

\subsubsection{Smoothing splines and model selection}
For each latent coordinate \(j\), a smoothing spline \(s_{j}(\mu)\) is fitted by minimizing the usual smoothing functional:
\begin{equation}
\min_{s\in\mathcal{C}^2([\mu_{\min},\mu_{\max}])}\ 
\mathcal{J}[s]
:=\sum_{i=1}^{N_p} w_i\bigl(y_i - s(\mu_i)\bigr)^2 \;+\; \alpha \int_{\mu_{\min}}^{\mu_{\max}} (s''(\mu))^2 \, \mathrm{d}\mu,
\label{eq:smoothing_functional}
\end{equation}
where \(y_i\) denotes the observed latent coordinate at parameter \(\mu_i\) (either \(Z_p(j,i)\) or \(\Phi_p(j,i)\)), \(w_i>0\) are sample weights, and \(\alpha\ge 0\) is the penalty weight. Here \(\mathcal{C}^2([\mu_{\min},\mu_{\max}])\) is the space of twice continuously differentiable functions on the closed parameter interval \([\mu_{\min},\mu_{\max}]\).

\paragraph{Equivalent discrete linear system}
Let \(\{B_m(\mu)\}_{m=1}^{M}\) be a suitable cubic B-spline basis (with knots at or between the sample locations). Writing the spline as \(s(\mu)=\sum_{m=1}^{M} c_m B_m(\mu)\) and collecting values in vector form \( \mathbf{c}=(c_1,\dots,c_{M})^\top\), the functional \(\mathcal{J}\) becomes a quadratic form
\begin{equation}
\mathcal{J}[\mathbf{c}] = (\mathbf{y}-\mathbf{B}\mathbf{c})^\top W (\mathbf{y}-\mathbf{B}\mathbf{c}) \;+\; \alpha \, \mathbf{c}^\top R \mathbf{c},
\end{equation}
where \(\mathbf{y}=(y_1,\dots,y_{N_p})^\top\), \(\mathbf{B}\in\mathbb{R}^{N_p\times M}\) with \(\mathbf{B}_{i,m}=B_m(\mu_i)\), \(W=\operatorname{diag}(w_1,\dots,w_{N_p})\), and \(R\in\mathbb{R}^{M\times M}\) are the positive semidefinite penalty matrices with entries \(R_{m,n}=\int B_m''(\mu) B_n''(\mu)\,\mathrm{d}\mu\). The minimizer \(\mathbf{c}^\star\) solves the symmetric linear system
\begin{equation}
\bigl(\mathbf{B}^\top W \mathbf{B} + \alpha R \bigr)\mathbf{c}^\star \;=\; \mathbf{B}^\top W \mathbf{y}.
\label{eq:spline_normal_equation}
\end{equation}
This linear-algebraic form clarifies numerical implementation and the role of the smoothing weight \(\alpha\). Here, \(M\) represents the number of basis functions used to represent the spline (i.e., the dimension of the B-spline basis). Typically, if nodes are placed at each sample, then \(M\approx N_p\), it is also possible to choose \(M<N_p\) (sparse nodes) to reduce the degrees of freedom and enhance smoothness.

\paragraph{Weighted K-fold cross validation on a candidate \(s\)-grid} 
The smoothing level is selected using \(K\)-fold cross validation over a finite candidate grid \(\mathcal{S}=\{s_1,\dots,s_L\}\), typically logarithmically spaced. Let \(\mathcal{I}=\{1,\dots,N_p\}\) denote the index set of samples, and \(\{V_k\}_{k=1}^K\) a partition of \(\mathcal{I}\) into \(K\) disjoint validation folds, with corresponding training sets \(T_k=\mathcal{I}\setminus V_k\). For each candidate \(s_\ell\in\mathcal{S}\) and each fold \(k\), the spline \(s^{(k,\ell)}(\mu)\) is fitted by solving the penalized problem \eqref{eq:smoothing_functional} restricted to the training indices \(T_k\); equivalently, \(\mathbf{B}_{T_k}\) and \(W_{T_k}=\operatorname{diag}(w_i)_{i\in T_k}\) are formed and the linear system is solved:
\begin{equation}
(\mathbf{B}_{T_k}^\top W_{T_k}\mathbf{B}_{T_k}+\alpha(\ell) R)\mathbf{c}^{(k,\ell)} = \mathbf{B}_{T_k}^\top W_{T_k}\mathbf{y}_{T_k},
\end{equation}
where \(\alpha(\ell)\) is the penalty corresponding to \(s_\ell\). The fitted spline is then evaluated on the held-out fold \(V_k\) to compute the validation error
\begin{equation}
E_{k,\ell} = \frac{\sum_{i\in V_k} w_i\bigl(y_i - s^{(k,\ell)}(\mu_i)\bigr)^2}{\sum_{i\in V_k} w_i}.
\label{eq:fold_error}
\end{equation}

The cross-validation estimate for \(s_\ell\) is the fold-average
\begin{equation}
\mathrm{CV}(s_\ell) \;=\; \frac{1}{K}\sum_{k=1}^K E_{k,\ell},
\end{equation}
and the parameter is selected as
\begin{equation}
s^\star \;=\; \arg\min_{s_\ell\in\mathcal{S}} \mathrm{CV}(s_\ell).
\label{eq:choose_s}
\end{equation}
Finally, the chosen spline \(s_j(\mu)\) is refitted on the full dataset \(\mathcal{I}\) using \(s^\star\) (or its corresponding \(\alpha^\star\)).

Denote by \(\mathcal{O}\subset\mathcal{I}\) the set of original high-fidelity samples and by \(\mathcal{P}=\mathcal{I}\setminus\mathcal{O}\) the recursively predicted samples. A convenient and robust weighting rule is
\begin{equation}
w_i = \begin{cases}
1, & i\in\mathcal{O},\\
\gamma, & i\in\mathcal{P},
\end{cases}
\qquad 0<\gamma<1,
\label{eq:weight_rule}
\end{equation}
which down-weights extrapolated/less-trusted samples.

The fold partition \(\{V_k\}\) should be stratified (if possible) to distribute predicted samples \(\mathcal{P}\) across folds so that validation statistics reflect both original and predicted-data performance. The candidate grid \(\mathcal{S}\) is conveniently taken to be logarithmically spaced, e.g. \(s_\ell = 10^{a_\ell}\) for \(a_\ell\in[\!-1,0]\) or a wider range if necessary. The numerical solution of \eqref{eq:spline_normal_equation} is stable due to the penalty \(\alpha R\); nevertheless, one should monitor the condition number of the matrix \(\mathbf{B}^\top W\mathbf{B}+\alpha R\) and scale basis functions if ill-conditioning appears.

As the smoothing threshold \(s\downarrow 0\) (equivalently \(\alpha\downarrow 0\)), the spline approaches an interpolant of the training data, whereas at \(s\uparrow\infty\) (or \(\alpha\uparrow\infty\)), the curvature penalty dominates, and the solution tends toward a low-degree polynomial (in the extreme, a linear fit). Cross validation balances these two regimes in a data-driven manner.

\subsubsection{Dense latent reconstruction.}
The parameter variable is now treated as continuous on the convex hull of the discrete samples. Let
\begin{equation}
\mu_{\min}=\min_{i}\mu_i,\qquad \mu_{\max}=\max_{i}\mu_i,
\end{equation}
and consider an arbitrary parameter value \(\mu^\ast\in[\mu_{\min},\mu_{\max}]\). Continuous latent maps are obtained using the smoothing splines fitted per latent coordinate
\begin{equation}
\hat z : [\mu_{\min},\mu_{\max}]\to\mathbb{R}^{r^p},\qquad
\hat\phi : [\mu_{\min},\mu_{\max}]\to\mathbb{R}^{r^p},
\end{equation}
with components \(\hat z(\mu)=[s_{z,1}(\mu),\dots,s_{z,r^p}(\mu)]^\top\) and \(\hat\phi(\mu)=[s_{\phi,1}(\mu),\dots,s_{\phi,r^p}(\mu)]^\top\).

The continuous linear reconstruction at an arbitrary parameter \(\mu^\ast\) is then
\begin{equation}
\widehat{U}_{\mathrm{lin}}(\mu^\ast)\;=\;\Psi_p^{(r_p)}\,\hat z(\mu^\ast)\in\mathbb{R}^{N},
\label{eq:continuous_lin_recon}
\end{equation}
where \(\Psi_p^{(r_p)}\in\mathbb{R}^{N\times r^p}\) is obtained in equation \eqref{parametric linear}.  
Similarly, the continuous nonlinear correction (residual) is obtained by applying the previously learned nonlinear map \(\mathcal{K}^{(p)}:\mathbb{R}^{r^p}\to\mathbb{R}^N\) obtained in equation \eqref{parametric nonlinear} to the continuous nonlinear coordinates:
\begin{equation}
\widehat{\Delta}(\mu^\ast)\;=\;\mathcal{K}^{(p)}\bigl(\hat\phi(\mu^\ast)\bigr)\in\mathbb{R}^N.
\label{eq:continuous_delta}
\end{equation}
Combining \eqref{eq:continuous_lin_recon} and \eqref{eq:continuous_delta} yields a \emph{continuous} parametric reconstruction operator
\begin{equation} 
\widehat U(\mu^\ast)\;=\;\widehat{U}_{\mathrm{lin}}(\mu^\ast)\;+\;\widehat{\Delta}(\mu^\ast)
\qquad\text{for all }\mu^\ast\in[\mu_{\min},\mu_{\max}].
\label{eq:continuous_combined}
\end{equation}

The smoothing-spline interpolation in latent space converts a discrete set of (possibly noisy) latent coordinates into a dense, smooth, and numerically stable representation over the parameter interval. A continuous reconstruction operator \(\mu\mapsto\widehat U(\mu)\) is thus obtained, which (i) enables accurate pointwise predictions at unseen parameter values and (ii) supplies dense, smooth parametric fields for downstream analysis, visualization, and surrogate-model training.

\subsection{Full-state reconstruction}

As in standard PMD, PPMD reconstructs the high-dimensional state by combining linear and nonlinear components:
\begin{equation}
\widehat{\mathbf{u}}(\mu) \;=\; \Psi_p^{(r_p)}\,\hat z(\mu) \;+\; \mathcal{K}^{(p)} \bigl(\hat\phi(\mu)\bigr),
\label{eq:ppmd_full_recon}
\end{equation}
where \(\Psi_p^{(r_p)}\in\mathbb{R}^{N\times r^p}\) denotes the PPMD truncated linear basis obtained from equation \eqref{parametric linear} in the PPMD linear stage, \(\hat z(\mu)\in\mathbb{R}^{r^p}\) and \(\hat\phi(\mu)\in\mathbb{R}^{r^p}\) are the continuous latent maps obtained by spline interpolation of the linear and nonlinear coordinates respectively, and \(\mathcal{K}^{(p)}:\mathbb{R}^{r^p}\to\mathbb{R}^N\) is the nonlinear lifting operator learned from residuals.

The nonlinear lifting \(\mathcal{K}^{(p)}\) is learned via \emph{kernel ridge regression} using a polynomial kernel. 
Let \(\Phi_p=[\phi_1,\dots,\phi_{n_s}]\in\mathbb{R}^{r^p\times n_s}\) \eqref{reductionresidual} be the matrix of nonlinear coordinates at the training parameter samples and \(R_p\in\mathbb{R}^{N\times n_s}\) \eqref{residual} the corresponding residuals. 
Define the kernel matrix \(\mathcal{M}\in\mathbb{R}^{n_s\times n_s}\) with entries
\begin{equation}
\mathcal{M}_{ij} \;=\; m(\phi_i,\phi_j) \;=\; \bigl(\phi_i^\top \phi_j + c\bigr)^d,
\end{equation}
where \(d\in\mathbb{N}\) is the polynomial degree and \(c\ge 0\) the offset (the usual polynomial-kernel parameters). \(\mathcal{K}^{(p)}\) is obtained as the minimizer of
\begin{equation}\label{parametric nonlinear}
\mathcal{K}^{(p)} \;=\; \arg\min_{\mathcal{K}^{(p)}}\; \bigl\| \mathcal{K}^{(p)}\,\mathcal{M} - R_p \bigr\|_F^2 \;+\; \lambda \,\|\mathcal{K}^{(p)}\|_F^2,
\end{equation}
where \(\lambda>0\) is a regularization parameter. In practice \(\mathcal{K}^{(p)}\) is implemented via a polynomial feature map followed by a KRR estimator (evaluated column-wise) so that \(\mathcal{K}^{(p)}(\hat\phi(\mu))\) returns the high-dimensional residual corresponding to the nonlinear coordinates \(\hat\phi(\mu)\).

PPMD constructs continuous parametric surrogates from limited high-fidelity data through a coherent pipeline: it first aggregates snapshots across parameter samples to form a parameter-driven model, then decomposes the state into a linear component \(\Psi_p^{(r)}\hat z(\mu)\) and a nonlinear residual \(\mathcal{K}(\hat\phi(\mu))\). Both latent coordinates are extrapolated via recursive adaptive kernel ridge regression for robust multi-step prediction. Finally, discrete latent samples are converted into smooth maps \(\hat z(\mu),\hat\phi(\mu)\) using weighted smoothing splines, yielding a continuous reconstruction operator \(\mu\mapsto\widehat U(\mu)\) for accurate pointwise evaluation and downstream tasks.

\section{Convergence analysis of the PPMD}
\label{sec:ppmd_convergence}

A convergence analysis is essential to validate the PPMD method, ensuring that the parametric reduced-order model reliably approximates the underlying solution manifold. This section provides a rigorous mathematical framework for assessing the accuracy and computational efficiency of PPMD, with particular emphasis on the convergence properties of its linear and nonlinear components.

\subsection{Convergence of linear parametric features}

Given a parametric family of solutions $\{U(\mu)\}_{\mu \in \mathcal{P}}$ where $\mathcal{P} = [\mu_{\min}, \mu_{\max}] \subset \mathbb{R}^p$ is a compact parameter domain. A set of $n_s$ parameter samples ${\mu_i}_{i=1}^{n_s}$ is drawn independently from a distribution $\rho$ on $\mathcal{P}$. For each $\mu_i$, the corresponding high-fidelity snapshot $\bar{u}(\mu_i) \in \mathbb{R}^N$ is obtained.

\subsubsection{Construction of the sample matrix and covariance operators}

The column-normalized parametric snapshot matrix is defined as follows:
\begin{equation}
\bar{U}_p = [\bar{u}(\mu_1), \dots, \bar{u}(\mu_{n_s})] \in \mathbb{R}^{N \times n_s}.
\end{equation}

The true population covariance operator $C_p: \mathbb{R}^N \to \mathbb{R}^N$ is defined as follows:
\begin{equation}
C_p = \mathbb{E}_{\mu \sim \rho}[\bar{u}(\mu) \bar{u}(\mu)^\top] = \int_{\mathcal{P}} \bar{u}(\mu) \bar{u}(\mu)^\top \, d\rho(\mu).
\end{equation}

Its empirical counterpart, the sample covariance matrix $\hat{C}_{n_s}$, is given by:
\begin{equation}
\hat{C}_{n_s} = \frac{1}{n_s} \sum_{i=1}^{n_s} \bar{u}(\mu_i) \bar{u}(\mu_i)^\top = \frac{1}{n_s} \bar{U}_p \bar{U}_p^\top.
\end{equation}

Interest lies in the low-rank approximation of $\bar{U}_p$ via truncated SVD:
\begin{equation}
\bar{U}_p^{(r_p)} = \Psi_p^{(r_p)} \Lambda_p^{(r_p)} (\Theta_p^{(r_p)})^\top,
\end{equation}
where $\Psi_p^{(r_p)} \in \mathbb{R}^{N \times r_p}$ contains the leading $r_p$ left singular vectors, $\Lambda_p^{(r_p)} \in \mathbb{R}^{r_p \times r_p}$ is the diagonal matrix of singular values, and $\Theta_p^{(r_p)} \in \mathbb{R}^{n_s \times r_p}$ contains the right singular vectors.

\subsubsection{Convergence of the sample covariance matrix in operator norm}

The convergence of $\hat{C}_{n_s}$ to $C_p$ in the operator norm is first established. Under the assumption that the snapshots $\bar{u}(\mu)$ are uniformly bounded in $L^2(\rho; \mathbb{R}^N)$, i.e., there exists $M > 0$ such that $\mathbb{E}[|\bar{u}(\mu)|_2^2] \leq M^2$, matrix concentration inequalities can be applied.

Let $X_i = \bar{u}(\mu_i)\bar{u}(\mu_i)^\top$ for $i=1,\dots,n_s$. These are i.i.d. random symmetric positive semidefinite matrices with $\mathbb{E}[X_i] = C_p$. Recall that for a matrix $A$, the operator norm (or spectral norm) $\|A\|_{\mathrm{op}}$ is defined as $\|A\|_{\mathrm{op}} = \sup_{\|x\|_2=1} \|Ax\|_2$. Each $X_i$ satisfies the boundedness condition:
\begin{equation}
\|X_i\|_{\mathrm{op}} = \|\bar{u}(\mu_i)\|_2^2 \leq R \quad \text{almost surely},
\end{equation}
for some constant $R > 0$. The variance of $X_i$ is controlled by:
\begin{equation}
\|\mathbb{E}[X_i^2]\|_{\mathrm{op}} = \|\mathbb{E}[\|\bar{u}(\mu_i)\|_2^2 \bar{u}(\mu_i)\bar{u}(\mu_i)^\top]\|_{\mathrm{op}} \leq R \|C_p\|_{\mathrm{op}}.
\end{equation}

Define $Y_i = X_i / n_s$, so that $\hat{C}_{n_s} - C_p = \sum_{i=1}^{n_s} Y_i$. Then $\|Y_i\|_{\mathrm{op}} \leq 2R/n_s$ and the total variance satisfies
\begin{equation}
\Big\|\sum_{i=1}^{n_s} \mathbb{E}[Y_i^2]\Big\|_{\mathrm{op}} = \frac{1}{n_s}\|\mathbb{E}[X_1^2]\|_{\mathrm{op}} \leq \frac{R\|C_p\|_{\mathrm{op}} + R^2}{n_s}.
\end{equation}
By applying the matrix Bernstein inequality\cite{tropp2012user} with parameters $L = 2R/n_s$ and $\sigma^2 = (R\|C_p\|_{\mathrm{op}} + R^2)/n_s$, the following bound is obtained for any $t > 0$:
\begin{equation}
\mathbb{P}\left(\|\hat{C}_{n_s} - C_p\|_{\mathrm{op}} \geq t\right) \leq 2N \exp\left(-\frac{n_s t^2/2}{R\|C_p\|_{\mathrm{op}} + R^2 + (2R/3)t}\right).
\end{equation}
To obtain a high-probability bound, it sets the right-hand side equal to $\delta$ and solves for $t$. Let $S = \log(2N/\delta)$. The inequality $2N \exp(-n_s t^2/(2V + 4Rt/3)) \leq \delta$ (where $V = R\|C_p\|_{\mathrm{op}} + R^2$) leads to the quadratic inequality $n_s t^2 - (4RS/3)t - 2SV \geq 0$. Its positive solution yields
\begin{equation}
t \leq \frac{4RS}{3n_s} + \sqrt{\frac{2SV}{n_s} + \left(\frac{4RS}{3n_s}\right)^2} \leq C\left(\sqrt{\frac{\log(2N/\delta)}{n_s}} + \frac{\log(2N/\delta)}{n_s}\right),
\end{equation}
$C$ depends on $R$ and $\|C_p\|_{\mathrm{op}}$. For large $n_s$, dominant term is $\mathcal{O}(\sqrt{\log(2N/\delta)/n_s})$. Hence, with a probability of at least $1-\delta$,
\begin{equation}
\|\hat{C}_{n_s} - C_p\|_{\mathrm{op}} \leq C \sqrt{\frac{\log(2N/\delta)}{n_s}}.
\end{equation}
To obtain the mean convergence rate, the tail bound is integrated:
\begin{equation}
\mathbb{E}[\|\hat{C}_{n_s} - C_p\|_{\mathrm{op}}] = \int_0^\infty \mathbb{P}(\|\hat{C}_{n_s} - C_p\|_{\mathrm{op}} \geq t) \, dt.
\end{equation}
Splitting the integral at $t_0 = C\sqrt{\log(2N)/n_s}$ and using the Bernstein bound for $t \geq t_0$, the following bound is obtained:
\begin{equation}
\mathbb{E}[\|\hat{C}_{n_s} - C_p\|_{\mathrm{op}}] \leq t_0 + \int_{t_0}^\infty 2N \exp\left(-\frac{n_s t^2/2}{V + (2R/3)t}\right) dt = \mathcal{O}\left(\sqrt{\frac{\log N}{n_s}}\right).
\end{equation}
Since $N$ (the spatial dimension) is fixed, it is concluded that
\begin{equation}
\mathbb{E}[\|\hat{C}_{n_s} - C_p\|_{\mathrm{op}}] = \mathcal{O}\left(\frac{1}{\sqrt{n_s}}\right).
\end{equation}
This establishes that the sample covariance matrix converges to the true covariance operator at a rate of $\mathcal{O}(1/\sqrt{n_s})$ in operator norm.

\subsubsection{Convergence of eigenvalues}

Let $\lambda_1 \geq \lambda_2 \geq \cdots \geq \lambda_N \geq 0$ denote the eigenvalues of $C_p$, and $\hat{\lambda}_1 \geq \hat{\lambda}_2 \geq \cdots \geq \hat{\lambda}_N \geq 0$ the eigenvalues of $\hat{C}_{n_s}$. By Weyl's inequality \cite{horn2012matrix}, the following bound holds for each $k = 1,\dots,N$:
\begin{equation}
|\hat{\lambda}_k - \lambda_k| \leq \|\hat{C}_{n_s} - C_p\|_{\mathrm{op}}.
\end{equation}

Therefore, with a probability of at least $1-\delta$,
\begin{equation}
\max_{1 \leq k \leq N} |\hat{\lambda}_k - \lambda_k| \leq C \sqrt{\frac{\log(2N/\delta)}{n_s}},
\end{equation}
and in expectation,
\begin{equation}
\mathbb{E}[|\hat{\lambda}_k - \lambda_k|] = \mathcal{O}\left(\frac{1}{\sqrt{n_s}}\right).
\end{equation}

This shows that the empirical eigenvalues converge uniformly to the population eigenvalues at the same $\mathcal{O}(1/\sqrt{n_s})$ rate.

\subsubsection{Convergence of eigenvectors and principal subspaces}

The convergence of eigenvectors requires additional assumptions on the eigenvalue gaps. Let $\delta_k = \min\{\lambda_{k-1} - \lambda_k, \lambda_k - \lambda_{k+1}\}$ denote the spectral gap around $\lambda_k$, with the convention $\lambda_0 = \infty$ and $\lambda_{N+1} = 0$. Assume $\delta_k > 0$ for $k = 1,\dots,r_p$.

Let $v_k$ and $\hat{v}_k$ be the $k$-th eigenvectors of $C_p$ and $\hat{C}_{n_s}$ respectively, corresponding to $\lambda_k$ and $\hat{\lambda}_k$. By the Davis-Kahan $\sin\Theta$ theorem \cite{davis1970rotation, yu2015useful}, the following bound holds:
\begin{equation}
\sin\angle(v_k, \hat{v}_k) \leq \frac{2\|\hat{C}_{n_s} - C_p\|_{\mathrm{op}}}{\delta_k}.
\end{equation}

Thus, with a probability of at least $1-\delta$,
\begin{equation}
\sin\angle(v_k, \hat{v}_k) \leq \frac{2C}{\delta_k} \sqrt{\frac{\log(2N/\delta)}{n_s}},
\end{equation}
and in expectation,
\begin{equation}
\mathbb{E}[\sin\angle(v_k, \hat{v}_k)] = \mathcal{O}\left(\frac{1}{\delta_k\sqrt{n_s}}\right).
\end{equation}

For the principal subspace spanned by the first $r_p$ eigenvectors, define the projection matrices:
\begin{equation}
P = \sum_{k=1}^{r_p} v_k v_k^\top, \quad \hat{P} = \sum_{k=1}^{r_p} \hat{v}_k \hat{v}_k^\top.
\end{equation}

The distance between these subspaces can be measured by the spectral norm of their difference. Applying the Davis-Kahan theorem for subspaces, the following bound is obtained:
\begin{equation}
\|\hat{P} - P\|_{\mathrm{op}} \leq \frac{2\sqrt{r_p} \|\hat{C}_{n_s} - C_p\|_{\mathrm{op}}}{\min_{1 \leq k \leq r_p} \delta_k}.
\end{equation}

Therefore, with a probability of at least $1-\delta$,
\begin{equation}
\|\hat{P} - P\|_{\mathrm{op}} \leq \frac{2\sqrt{r_p} C}{\min_{1 \leq k \leq r_p} \delta_k} \sqrt{\frac{\log(2N/\delta)}{n_s}},
\end{equation}
and in expectation,
\begin{equation}
\mathbb{E}[\|\hat{P} - P\|_{\mathrm{op}}] = \mathcal{O}\left(\frac{\sqrt{r_p}}{\min_{1 \leq k \leq r_p} \delta_k \sqrt{n_s}}\right).
\end{equation}

This establishes the convergence of the linear subspace spanned by the leading $r_p$ singular vectors $\Psi_p^{(r_p)}$ to the true principal subspace of the covariance operator $C_p$. The convergence rate is $\mathcal{O}(1/\sqrt{n_s})$, modulated by the spectral gaps and subspace dimension.

\subsubsection{Implications for PPMD linear reconstruction}

In PPMD, the linear reconstruction uses the estimated basis $\Psi_p^{(r_p)}$ and the corresponding latent coordinates $z(\mu)$. The convergence of $\Psi_p^{(r_p)}$ ensures that for any parameter $\mu$, the projection error is:
\begin{equation}
\|P \bar{u}(\mu) - \hat{P} \bar{u}(\mu)\|_2 \leq \|\hat{P} - P\|_{\mathrm{op}} \|\bar{u}(\mu)\|_2
\end{equation}
vanishes as $n_s \to \infty$. Consequently, the linear component of PPMD provides a consistent approximation of the dominant linear features of the parametric solution manifold.

The analysis presented here provides a rigorous foundation for the linear part of PPMD. Combined with the convergence analysis of the nonlinear components (spectral embedding and kernel ridge regression), it ensures the overall consistency and reliability of the PPMD framework for parametric reduced-order modeling.

\subsection{Convergence of nonlinear parametric features}

The nonlinear residuals $R_p(\mu) = \bar{u}(\mu) - \Psi_p^{(r_p)}z(\mu)$ lie on a low-dimensional manifold $\mathcal{M} \subset \mathbb{R}^N$. PPMD constructs a probabilistic transition operator on $\mathcal{M}$ to extract nonlinear coordinates $\phi(\mu)$. This section provides a rigorous convergence analysis for both the probabilistic transition operator construction and the subsequent kernel ridge regression.

\subsubsection{Convergence of the PPMD operator}

\begin{lemma}[Graph-geodesic approximation]\label{lem:graph_geodesic}
Assuming the residual vectors $\{r_i\}_{i=1}^{n_s}$ are i.i.d. samples from a compact $d$-dimensional Riemannian manifold $\mathcal{M} \subset \mathbb{R}^N$ with a smooth density $p: \mathcal{M} \to (0,\infty)$ satisfying $0 < p_{\min} \leq p(x) \leq p_{\max} < \infty$ for all $x \in \mathcal{M}$. Let $d_{\mathcal{M}}(\cdot,\cdot)$ denote the geodesic distance on $\mathcal{M}$. In practice, this distance is approximated using graph shortest-path distances $\operatorname{dist}_G(r_i, r_j)$ computed on a $k$-nearest neighbor graph $G$ with $k = k(n_s) \asymp \log n_s$. For a $k$-NN graph with $k \asymp \log n_s$, the following bound holds with probability at least $1 - n_s^{-1}$:
\begin{equation}
\sup_{1 \leq i,j \leq n_s} |\operatorname{dist}_G(r_i, r_j) - d_{\mathcal{M}}(r_i, r_j)| = \mathcal{O}\left(\left(\frac{\log n_s}{n_s}\right)^{1/d}\right).
\end{equation}
\end{lemma}
\begin{proof}
The proof follows from standard covering arguments in Riemannian geometry \cite{bernstein2000graph, tenenbaum2000global}. Since $\mathcal{M}$ is compact and $d$-dimensional, the minimum number of balls of radius $\delta$ needed to cover $\mathcal{M}$ is $\mathcal{O}(\delta^{-d})$. With $n_s$ i.i.d. samples, the covering radius is $\mathcal{O}((\log n_s/n_s)^{1/d})$ with high probability. For any two points $r_i, r_j$, construct a path along the geodesic connecting them and place intermediate points at a spacing of $\delta_n = \mathcal{O}((\log n_s/n_s)^{1/d})$. By the covering property, each intermediate point is within $\delta_n$ of some sample point. The graph path connecting these samples approximates the geodesic with an error of $\mathcal{O}(\delta_n)$.
\end{proof}

Define the affinity matrix $A \in \mathbb{R}^{n_s \times n_s}$ with entries:
\begin{equation}
A_{ij} = \exp\left(-\frac{\operatorname{dist}_G(r_i, r_j)^2}{\varepsilon_p^2}\right),
\end{equation}
where $\varepsilon_p = \varepsilon_p(n_s) > 0$ is a bandwidth parameter. Let $D = \operatorname{diag}(D_1, \dots, D_{n_s})$ with $D_i = \sum_{j=1}^{n_s} A_{ij}$, and define the normalized probabilistic transition matrix:
\begin{equation}
P_p = D^{-1}A.
\end{equation}

The continuous analog is the kernel integral operator $\mathcal{P}_{\varepsilon_p}: L^2(\mathcal{M}) \to L^2(\mathcal{M})$ defined by:
\begin{equation}
(\mathcal{P}_{\varepsilon_p} f)(x) = \frac{\int_{\mathcal{M}} \exp\left(-\frac{d_{\mathcal{M}}(x,y)^2}{\varepsilon_p^2}\right) f(y) p(y) \, dV(y)}{\int_{\mathcal{M}} \exp\left(-\frac{d_{\mathcal{M}}(x,y)^2}{\varepsilon_p^2}\right) p(y) \, dV(y)},
\end{equation}
where $dV$ is the volume measure on $\mathcal{M}$.

\begin{theorem}[Convergence of the probability transition operator]\label{thm:diffusion_convergence}
Under the assumption in ~\ref{lem:graph_geodesic} and with $\varepsilon_p = \varepsilon_p(n_s)$ satisfying $\varepsilon_p \to 0$ and $n_s \varepsilon_p^{d/2} \to \infty$ as $n_s \to \infty$, let \(\delta \in (0,1)\) be a confidence parameter, then the following bound holds with probability at least $1 - \delta$:
\begin{equation}
\|P_p - \mathcal{P}_{\varepsilon_p}\|_{\mathrm{op}} \leq C_1 \varepsilon_p^2 + C_2 \sqrt{\frac{\log(1/\delta)}{n_s \varepsilon_p^{d/2}}} + C_3 \frac{(\log n_s/n_s)^{1/d}}{\varepsilon_p},
\end{equation}
where $C_1, C_2, C_3 > 0$ depends on $d$, $p_{\min}$, $p_{\max}$, and the geometry of $\mathcal{M}$.
\end{theorem}

\begin{proof}
The proof decomposes the error into three components:
\begin{enumerate}
    \item \textbf{Bias term ($\mathcal{O}(\varepsilon_p^2)$)}: Using the Taylor expansion of the kernel, one can show \cite{coifman2006diffusion} that for smooth $f$,
    \begin{equation}
    \mathcal{P}_{\varepsilon_p} f(x) = f(x) + \frac{\varepsilon_p^2}{2} \left( \Delta_{\mathcal{M}} f(x) + 2\frac{\nabla p(x) \cdot \nabla f(x)}{p(x)} \right) + \mathcal{O}(\varepsilon_p^4),
    \end{equation}
    where $\Delta_{\mathcal{M}}$ is the Laplace-Beltrami operator. This yields $\|\mathcal{P}_{\varepsilon_p} - I\|_{\mathrm{op}} = \mathcal{O}(\varepsilon_p^2)$.

    \item \textbf{Variance term ($\mathcal{O}(1/\sqrt{n_s\varepsilon_p^{d/2}})$)}: For fixed $x \in \mathcal{M}$, the empirical estimate of $\int_{\mathcal{M}} k_{\varepsilon_p}(x,y) f(y) p(y) dV(y)$ using $n_s$ samples has variance $\mathcal{O}(1/(n_s\varepsilon_p^{d/2}))$. Applying matrix Bernstein inequalities \cite{tropp2012user} to the kernel matrix yields the concentration bound.

    \item \textbf{Graph approximation term ($\mathcal{O}((\log n_s/n_s)^{1/d}/\varepsilon_p)$)}: By Lemma~\ref{lem:graph_geodesic} and the Lipschitz property of the Gaussian kernel $|e^{-a^2} - e^{-b^2}| \leq 2|a-b|/\varepsilon_p$ for $a,b \geq 0$, the error from using $\operatorname{dist}G$ instead of $d{\mathcal{M}}$ is bounded.
\end{enumerate}
Combining these three bounds via the triangle inequality gives the result.
\end{proof}

\begin{corollary}[Eigenvector convergence]
Let $(\lambda_k, \phi_k)$ be the $k$-th eigenvalue-eigenfunction pair of $\mathcal{P}_{\varepsilon_p}$, and $(\hat{\lambda}_k, \hat{\phi}_k)$ the corresponding pair for $P_p$. Assuming a spectral gap $\gamma_k = \min_{j \neq k} |\lambda_k - \lambda_j| > 0$, the Davis-Kahan theorem \cite{yu2015useful} yields:
\begin{equation}
\|\hat{\phi}_k - \phi_k\|_{L^2(\mathcal{M})} \leq \frac{2\|P_p - \mathcal{P}_{\varepsilon_p}\|_{\mathrm{op}}}{\gamma_k}.
\end{equation}
Thus, with the optimal bandwidth choice $\varepsilon_p \asymp n_s^{-2/(d+8)}$, a convergence rate of $\mathcal{O}_p(n_s^{-4/(d+8)})$ is obtained for the eigenvectors.
\end{corollary}

\subsubsection{Convergence of kernel ridge regression}

The nonlinear lifting $\mathcal{K}^{(p)}: \mathbb{R}^{r_p} \to \mathbb{R}^N$ is learned via KRR. For each output dimension $j = 1,\dots,N$, the following problem is solved:
\begin{equation}
\widehat{f}_j = \arg\min_{f \in \mathcal{H}_k} \left\{ \frac{1}{n_s} \sum_{i=1}^{n_s} |f(\phi_i) - r_{i,j}|^2 + \lambda \|f\|_{\mathcal{H}_k}^2 \right\},
\end{equation}
where $\phi_i \in \mathbb{R}^{r_p}$ are the estimated nonlinear coordinates, $r_{i,j}$ is the $j$-th component of the residual $r_i$, $\mathcal{H}_k$ is the reproducing kernel Hilbert space (RKHS) induced by a kernel $k$, and $\lambda > 0$ is the regularization parameter.

\begin{theorem}[Convergence of kernel ridge regression]\label{thm:krr_convergence}
Let the true nonlinear lifting function be denoted by $f_j^\dagger(\phi) = (\mathcal{K}^{(p)}(\phi))_j$ for $j = 1,\dots,N$. Assume that $f_j^\dagger$ is bounded, i.e., there exists $B > 0$ such that $\sup_{\phi \in \Phi} |f_j^\dagger(\phi)| \leq B$, and that it satisfies the source condition $f_j^\dagger = T_k^{\beta} g_j$ for some $g_j$ in the RKHS $\mathcal{H}_k$ with $\|g_j\|_{\mathcal{H}_k} \leq G$, where $\beta > 0$ and $T_k$ are the kernel integral operators defined by $[T_k f](\cdot) = \int_{\Phi} k(\cdot, \phi) f(\phi) d\rho(\phi)$. The kernel $k: \Phi \times \Phi \to \mathbb{R}$ is positive definite and bounded with $\sup_{\phi \in \Phi} k(\phi,\phi) \leq \kappa^2$, and the eigenvalues $\mu_\ell$ of $T_k$ decay polynomially as $\mu_\ell \asymp \ell^{-1/\tau}$ for some $\tau > 0$. The residuals satisfy $r_{i,j} = f_j^\dagger(\phi_i) + \xi_{i,j}$, where $\xi_{i,j}$ are independent zero-mean sub-Gaussian random variables with variance proxy $\sigma^2$, and the samples $\phi_i$ are drawn i.i.d. from a probability measure $\rho$ on $\Phi$.

Let $\widehat{f}_j$ be the KRR estimator with a regularization parameter $\lambda = \lambda(n_s) > 0$. Then for any $\delta \in (0,1)$, with probability at least $1 - \delta$,
\begin{equation}
\|\widehat{f}_j - f_j^\dagger\|_{L^2(\rho)} \leq C \left( \lambda^{\beta} + \frac{\sqrt{\log(1/\delta)}}{\lambda^{\tau}\sqrt{n_s}} \right),
\end{equation}
where $C > 0$ is a constant depending on $B, G, \kappa, \sigma$, and the spectral constants of $T_k$. In particular, choosing $\lambda \asymp n_s^{-1/(2(\beta+\tau))}$ yields the optimal convergence rate:
\begin{equation}
\|\widehat{f}_j - f_j^\dagger\|_{L^2(\rho)} = \mathcal{O}_p\left( n_s^{-\frac{\beta}{2(\beta+\tau)}} \sqrt{\log(1/\delta)} \right).
\end{equation}
\end{theorem}

\begin{proof}
The proof follows the standard bias-variance decomposition for kernel ridge regression \cite{caponnetto2007optimal}. Define the population risk minimizer:
\begin{equation}
f_\lambda = \arg\min_{f \in \mathcal{H}_k} \left\{ \|f - f_j^\dagger\|_{L^2(\rho)}^2 + \lambda \|f\|_{\mathcal{H}_k}^2 \right\}.
\end{equation}
The error decomposes as $\|\widehat{f}_j - f_j^\dagger\| \leq \|\widehat{f}_j - f_\lambda\| + \|f_\lambda - f_j^\dagger\|$.

\textbf{Bias term}: From the source condition $f_j^\dagger = T_k^{\beta} g_j$, the following bound is obtained:
\begin{equation}
\|f_\lambda - f_j^\dagger\|_{L^2(\rho)} \leq \|\lambda(T_k + \lambda I)^{-1} f_j^\dagger\|_{\mathcal{H}_k} \leq G \lambda^{\beta}.
\end{equation}

\textbf{Variance term}: The empirical error $\|\widehat{f}_j - f_\lambda\|$ can be bounded using concentration inequalities for operators in RKHS. Following \cite{smale2007learning}, define the empirical operator $\widehat{T}_k = \frac{1}{n_s} \sum_{i=1}^{n_s} k_{\phi_i} \otimes k_{\phi_i}$ and its population counterpart $T_k$. Then:
\begin{equation}
\widehat{f}_j - f_\lambda = (\widehat{T}_k + \lambda I)^{-1} \left( \frac{1}{n_s} \sum_{i=1}^{n_s} (r_{i,j} - f_\lambda(\phi_i)) k_{\phi_i} \right) - \lambda(T_k + \lambda I)^{-1} f_\lambda.
\end{equation}
Applying matrix Bernstein inequalities to $\widehat{T}_k - T_k$ and the noise term yields the variance bound $\mathcal{O}(1/(\lambda^{\tau}\sqrt{n_s}))$ with high probability.

Combining bias and variance gives the stated bound.
\end{proof}

\begin{corollary}[Vector-valued KRR convergence]
For the full vector-valued lifting $\mathcal{K}^{(p)}$, the following is obtained:
\begin{equation}
\|\widehat{\mathcal{K}}^{(p)} - \mathcal{K}^{(p)}\|_{\mathcal{L}} := \sup_{\|\phi\|_2=1} \|\widehat{\mathcal{K}}^{(p)}(\phi) - \mathcal{K}^{(p)}(\phi)\|_2 = \mathcal{O}_p\left( \sqrt{N} n_s^{-\frac{\beta}{2(\beta+\tau)}} \sqrt{\log(N/\delta)} \right).
\end{equation}
\end{corollary}
\begin{proof}
Apply Theorem~\ref{thm:krr_convergence} to each output dimension $j = 1,\dots,N$ and take a union bound. The factor $\sqrt{N}$ arises from summing the squared errors over dimensions.
\end{proof}

The combination of Theorem~\ref{thm:diffusion_convergence} and Theorem~\ref{thm:krr_convergence} ensures that the nonlinear component of PPMD consistently recovers the underlying manifold structure and provides accurate predictions for the residual field across the parameter space.

\subsection{Overall convergence of PPMD}

Combining linear and nonlinear components, the PPMD reconstruction operator $\widehat{\mathcal{R}}_{PPMD}(\mu)$ converges to the true solution $U(\mu)$ as sample sizes increase. The error decomposes as:
\begin{equation}
\|\widehat{\mathcal{R}}_{PPMD}(\mu) - U(\mu)\|_2 \leq E_{\mathrm{linear}} + E_{\mathrm{nonlinear}} + E_{\mathrm{repr}},
\end{equation}
where:
\begin{itemize}
\item $E_{\mathrm{linear}} = \mathcal{O}(n_s^{-1/2})$ from linear subspace approximation,
\item $E_{\mathrm{nonlinear}} = \mathcal{O}(n_s^{-2/(d+8)})$ from manifold learning and KRR,
\item $E_{\mathrm{repr}}$ is the irreducible representation error.
\end{itemize}

Thus, with optimal parameter scaling $\varepsilon_p \asymp n_s^{-2/(d+8)}$, PPMD achieves consistent parametric reduced-order modeling, converging to the best approximation within the chosen model class.

\section{Illustrative numerical examples}\label{sec:experiments}
In this section, the capability of the PPMD framework is demonstrated using two viscosity-parametrized flow problems:  2D past a circular cylinder and a Backward facing step. Full-order solutions for these parameter-dependent problems are generated using the finite element fluid solver Fluidity \cite{pain2005three}.  

For both test cases, unstructured triangular meshes are employed with sufficient resolution to ensure accurate solutions across the entire parameter range. The high-fidelity snapshot data obtained from these simulations are used to construct the PPMD reduced-order model, which is subsequently applied for prediction and reconstruction tasks across the parameter domain.  The performance of the PPMD is compared with a conventional POD-based ROM to highlight the advantages of manifold learning and probabilistic augmentation techniques.


\subsection{Case 1: Flow past a cylinder}

The \textit{flow past a cylinder} serves as a benchmark problem, offering a canonical setting to study parameter-dependent flow separation, vortex shedding, and wake dynamics. 

The computational domain is a two-dimensional channel of size \(1.8 \times 0.41\), with a circular cylinder of radius \(R = 0.05\) centered at \((0.2, 0.2)\). The flow is governed by the incompressible Navier–Stokes equations with constant density \(\rho = 1\). The dynamic viscosity \(\mu\) is treated as a varying parameter, ranging from \(\mu_{\min} = 2.0 \times 10^{-5}\) to \(\mu_{\max} = 1.01 \times 10^{-3}\) with a uniform increment of \(1 \times 10^{-5}\), yielding 100 distinct parameter values. The corresponding Reynolds numbers range from approximately \(\mathrm{Re} \approx 100\) to \(\mathrm{Re} \approx 5000\), where \(\mathrm{Re} = \rho U D / \mu\), \(U = 1.0\) is the inflow velocity, and \(D = 2R = 0.1\) is the cylinder diameter.

\begin{figure}[tbhp]
\centering
\begin{tabular}{cc}
\begin{minipage}{0.45\linewidth}
\includegraphics[width=\linewidth]{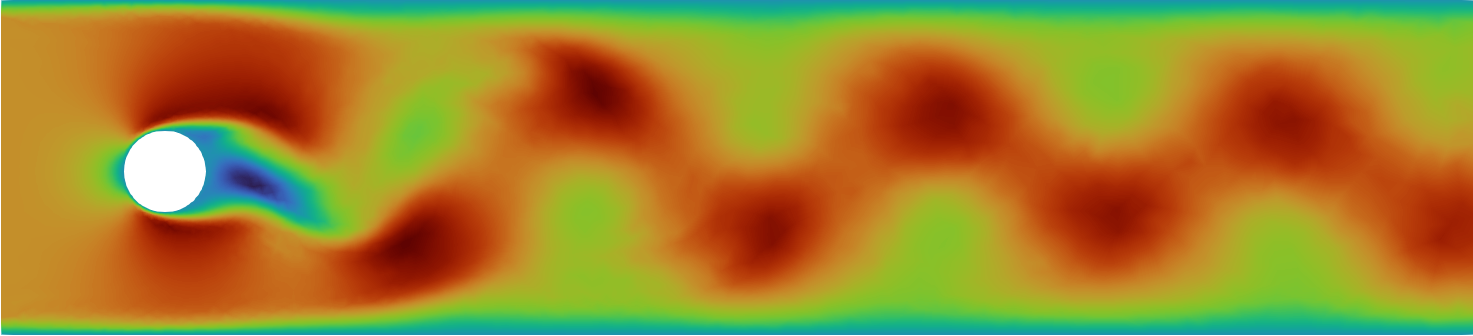}
\end{minipage} &
\begin{minipage}{0.45\linewidth}
\includegraphics[width=\linewidth]{flow4e-4full.png}
\end{minipage}\\
(a) {\small Full model, $t = 15s, Re = 250$} & (b) {\small Full model, $t = 15s, Re = 250$} \\
\begin{minipage}{0.45\linewidth}
\includegraphics[width=\linewidth]{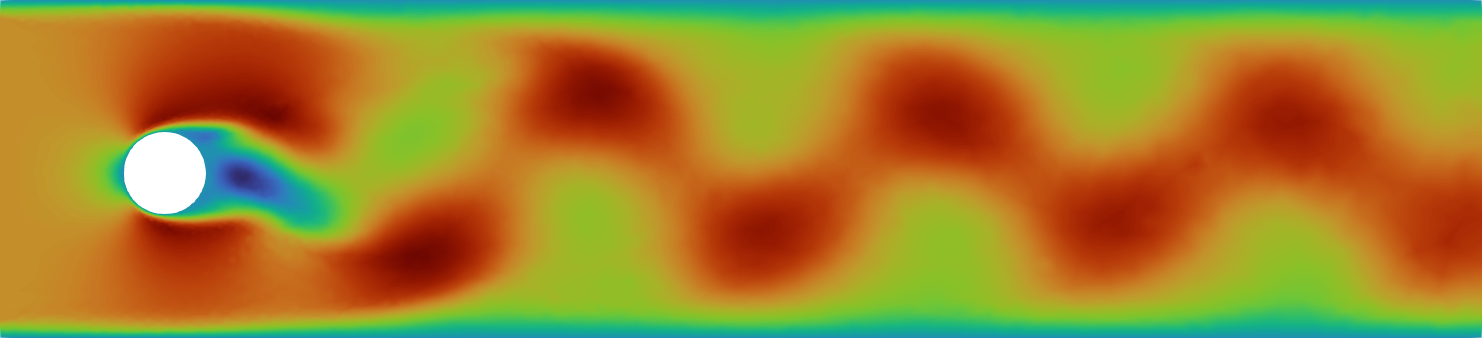}
\end{minipage} &
\begin{minipage}{0.45\linewidth}
\includegraphics[width=\linewidth]{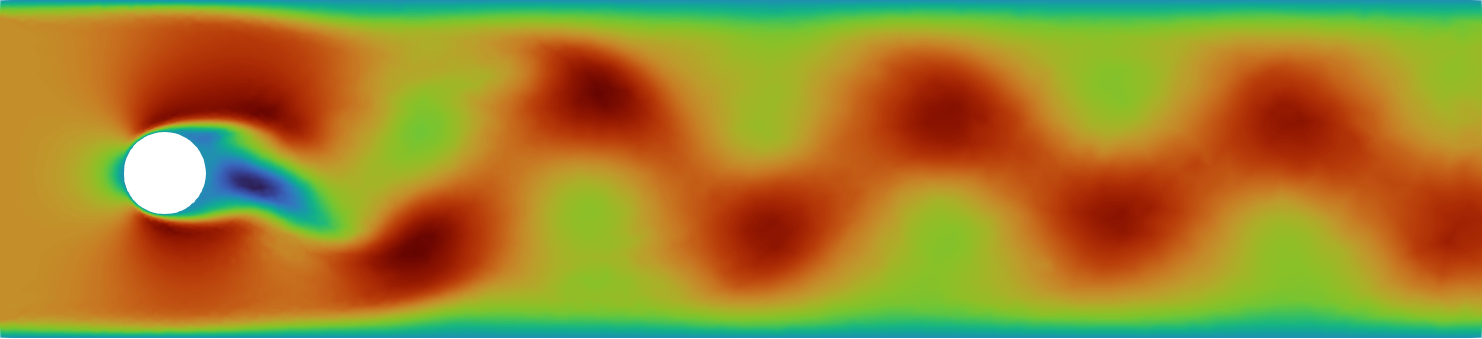}
\end{minipage} \\
(c) {\small POD+GPR, $r = 6$} & (d) {\small POD+GPR, $r = 12$} \\
\begin{minipage}{0.45\linewidth}
\includegraphics[width=\linewidth]{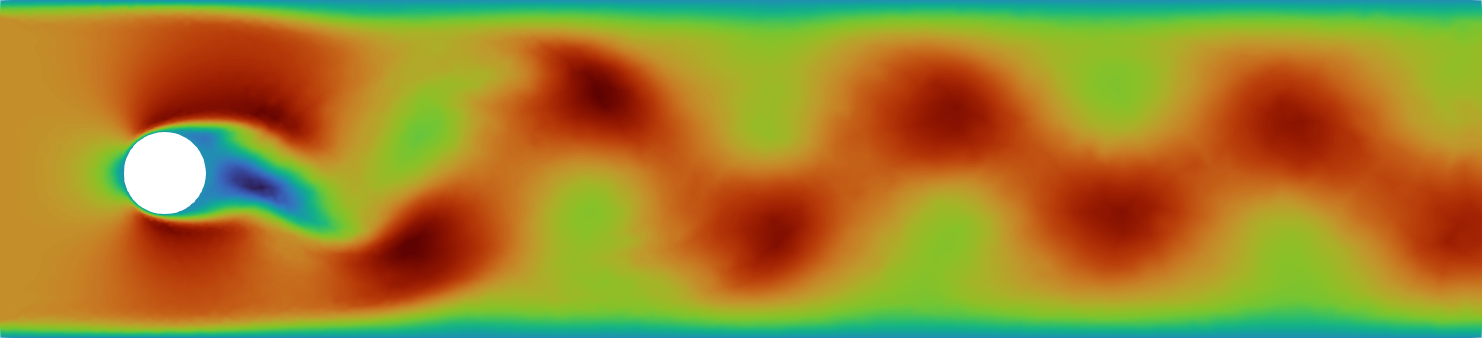}
\end{minipage} &
\begin{minipage}{0.45\linewidth}
\includegraphics[width=\linewidth]{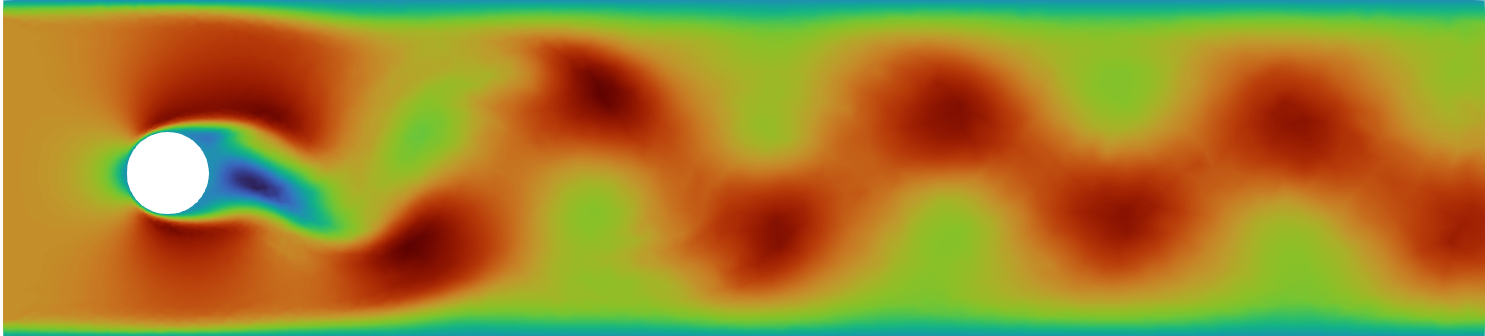}
\end{minipage} \\
(e) {\small PPMD, $r = 6$} & (f) {\small PPMD, $r = 12$} \\
\begin{minipage}{0.45\linewidth} 
\includegraphics[width = \linewidth,angle=0,clip=true]{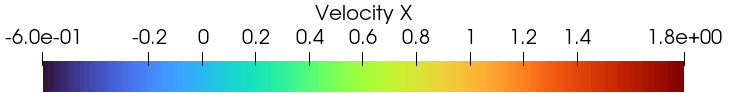}
\end{minipage}
&
\begin{minipage}{0.45\linewidth} 
\includegraphics[width = \linewidth,angle=0,clip=true]{scale1.png}
\end{minipage}\\
\end{tabular}
\caption{Flow past a cylinder at Re=250: Comparison of velocity reconstruction solutions at $t=15s$ from the high-fidelity full model, POD+GPR, and PPMD using 6 and 12 basis functions.}
\label{fg:flowreconstruct}
\end{figure}

\begin{figure}[tbhp]
\centering
\begin{tabular}{cc}
\begin{minipage}{0.45\linewidth}
\includegraphics[width=\linewidth]{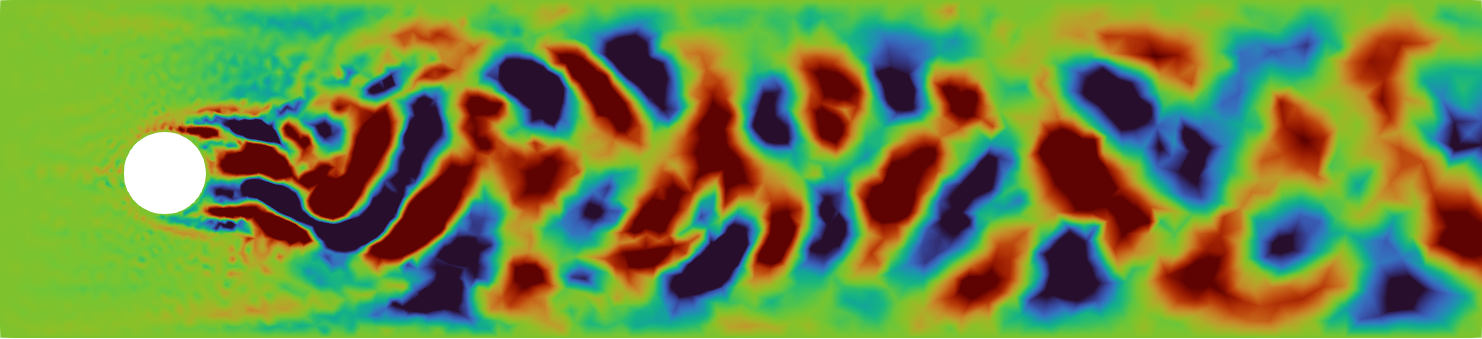}
\end{minipage} &
\begin{minipage}{0.45\linewidth}
\includegraphics[width=\linewidth]{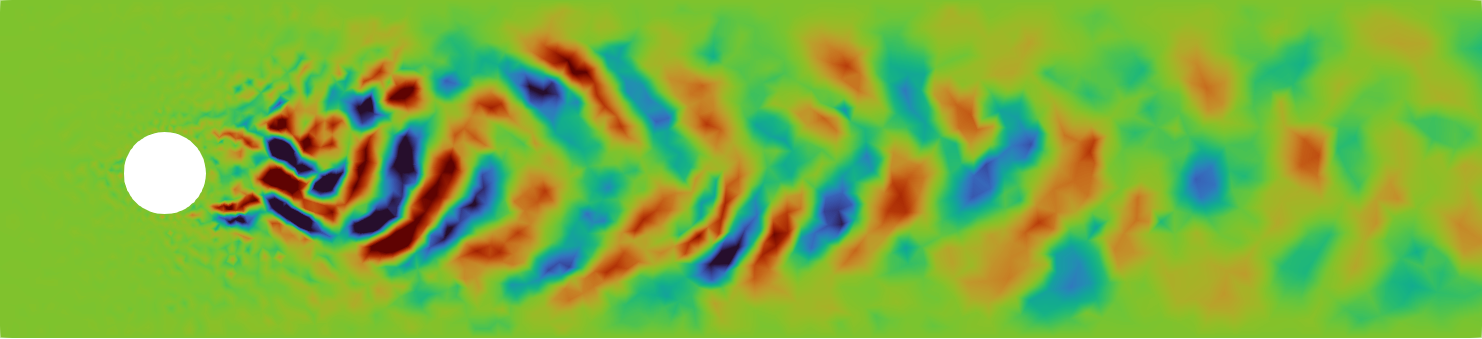}
\end{minipage} \\
(a) {\small POD+GPR, $r = 6$} & (b) {\small POD+GPR, $r = 12$} \\

\begin{minipage}{0.45\linewidth}
\includegraphics[width=\linewidth]{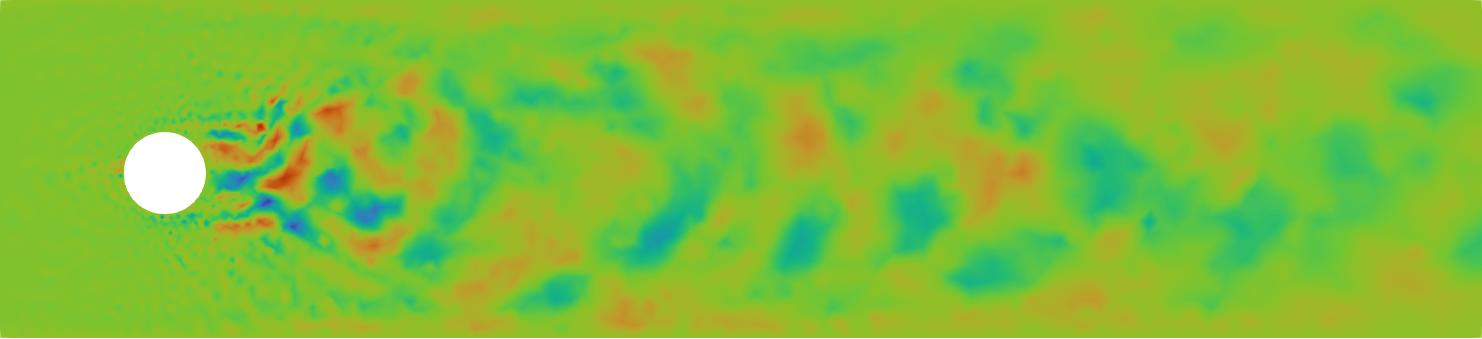}
\end{minipage} &
\begin{minipage}{0.45\linewidth}
\includegraphics[width=\linewidth]{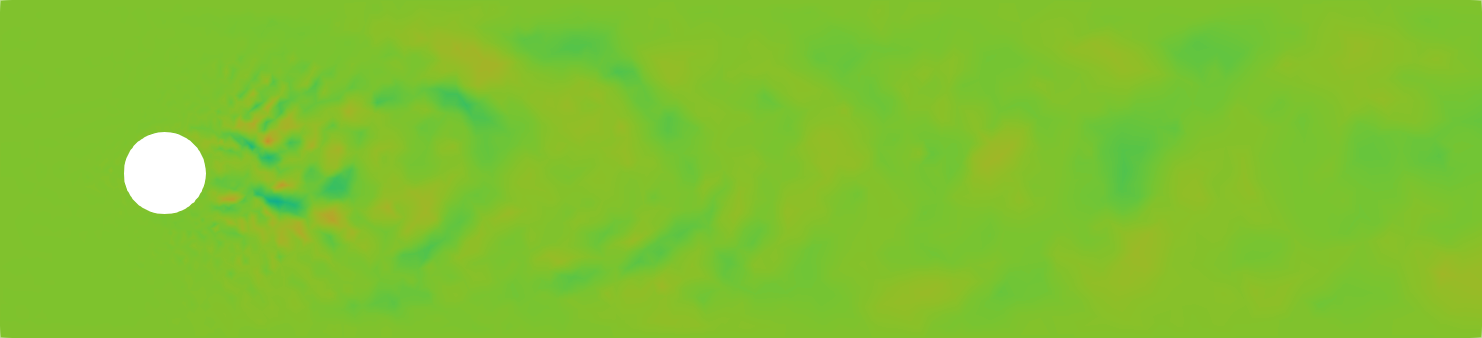}
\end{minipage} \\
(c) {\small PPMD, $r = 6$} & (d) {\small PPMD, $r = 12$} \\

\begin{minipage}{0.45\linewidth} 
\includegraphics[width = \linewidth,angle=0,clip=true]{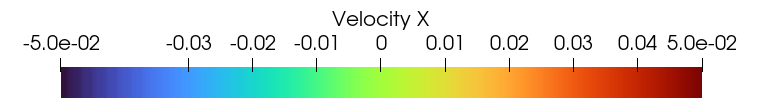}
\end{minipage}
&
\begin{minipage}{0.45\linewidth} 
\includegraphics[width = \linewidth,angle=0,clip=true]{scale2.png}
\end{minipage}\\

\end{tabular}
\caption{Flow past a cylinder: reconstruction errors at Re=250 and $t=15s$ for POD+GPR and PPMD using 6 and 12 basis functions.}
\label{fg:flowerror1}
\end{figure}

All simulations are initialized with a uniform velocity field \((u, v) = (0.1, 0.1)\). A Dirichlet inflow condition \((u, v) = (1.0, 0.0)\) is imposed at the left boundary, while no-slip conditions \((u, v) = (0, 0)\) are applied on the top and bottom walls and the cylinder surface. A pressure-Poisson-based outflow condition is used at the right boundary.

For each parameter value, the system is evolved over the time interval \(t \in [0, 15]\) seconds with a fixed time step \(\Delta t = 0.1\), resulting in 150 snapshots per simulation. The spatial discretization uses a mesh with 3802 nodes. The full dataset thus comprises 100 parameter samples, each with 150 snapshots. For training the PPMD model, the first 90 parameter values and their associated snapshots are used. The remaining 10 parameter values are reserved for out-of-sample prediction and error assessment. Additionally, a cubic smoothing spline is fitted to the 100 parameter-snapshot sets to obtain a continuous representation of the solution over the entire parameter interval, enabling reconstruction and prediction at arbitrary parameter values within the range.

\begin{figure}[tbhp]
\centering
\begin{tabular}{cc}
\begin{minipage}{0.45\linewidth}
\includegraphics[width=\linewidth]{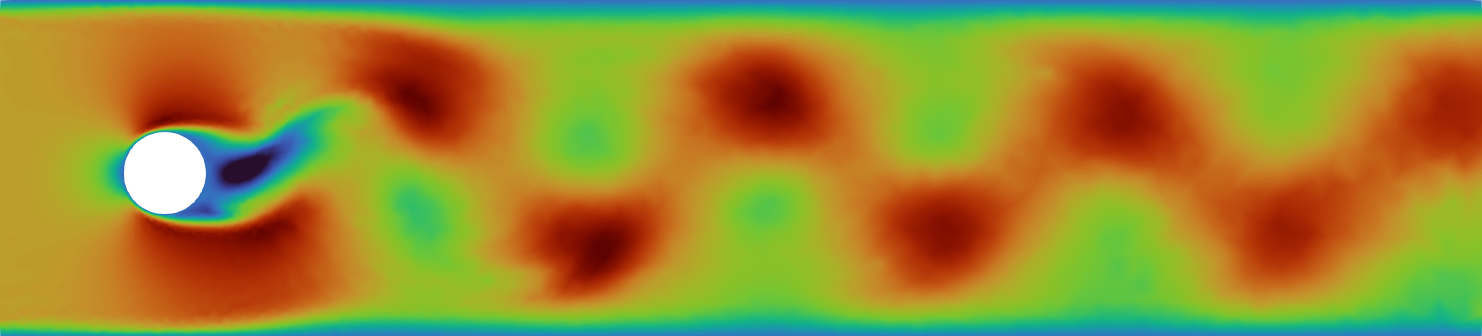}
\end{minipage} &
\begin{minipage}{0.45\linewidth}
\includegraphics[width=\linewidth]{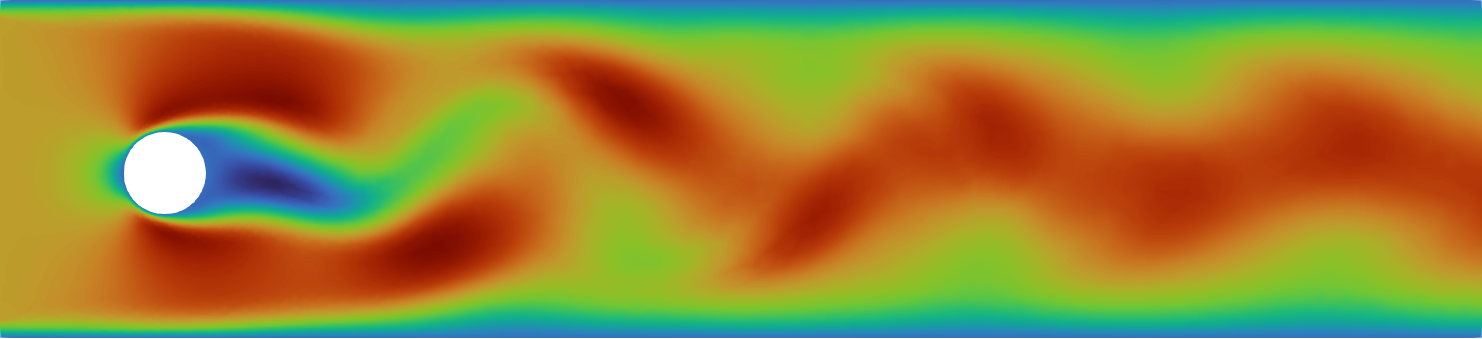}
\end{minipage} \\
(a) {\small Full model, $t = 15s, Re = 392$} & (b) {\small Full model, $t = 15s, Re = 104$} \\

\begin{minipage}{0.45\linewidth}
\includegraphics[width=\linewidth]{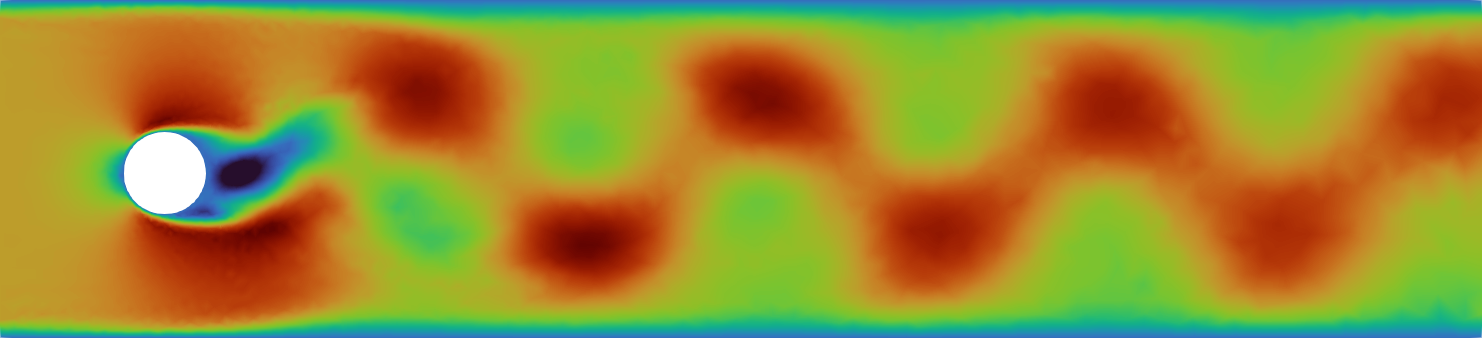}
\end{minipage} &
\begin{minipage}{0.45\linewidth}
\includegraphics[width=\linewidth]{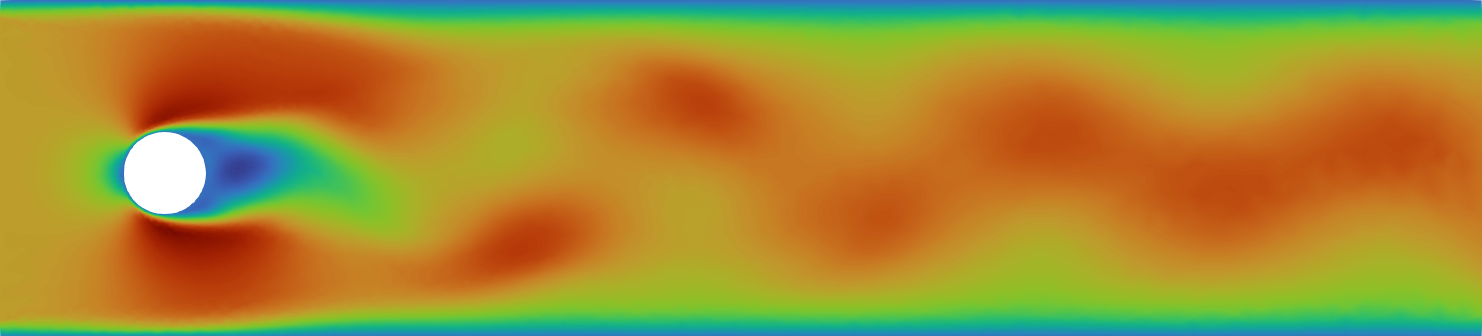}
\end{minipage} \\
(c) {\small POD+GPR, $r = 6, Re = 392$} & (d) {\small POD+GPR, $r = 6, Re = 104$} \\

\begin{minipage}{0.45\linewidth}
\includegraphics[width=\linewidth]{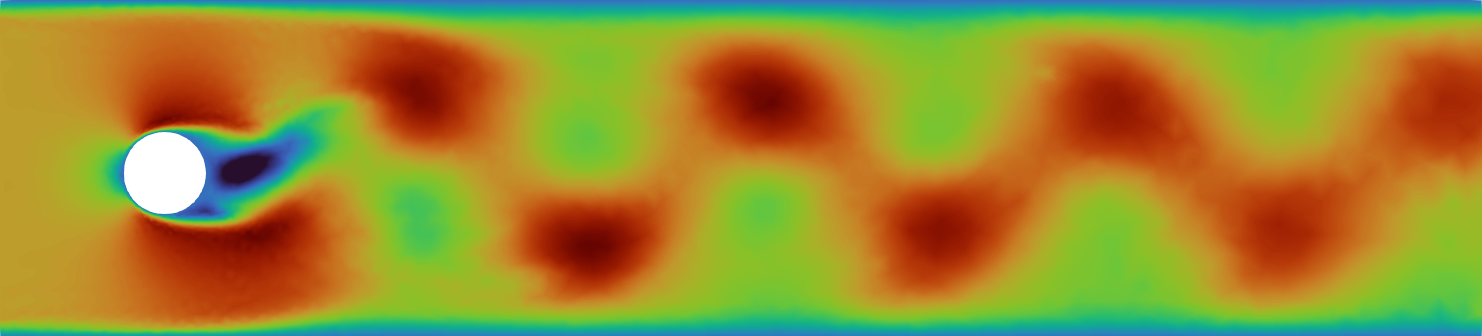}
\end{minipage} &
\begin{minipage}{0.45\linewidth}
\includegraphics[width=\linewidth]{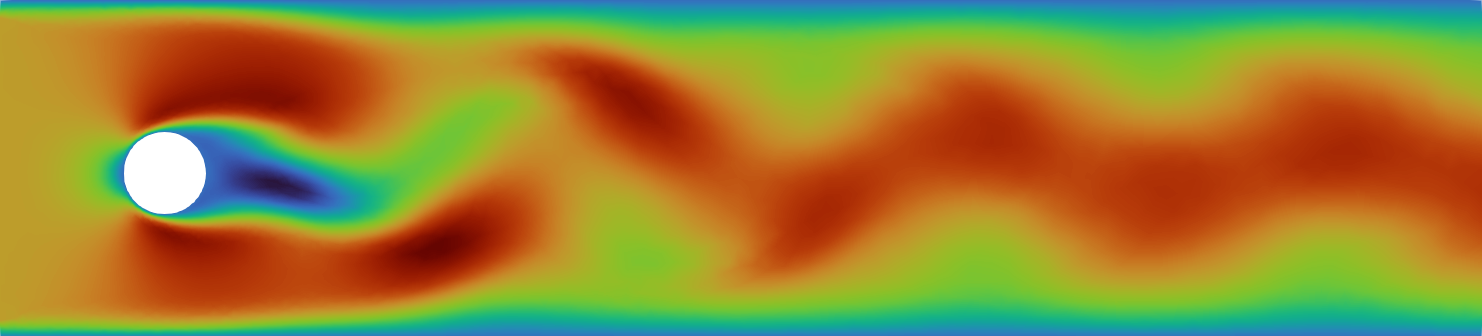}
\end{minipage} \\
(e) {\small PPMD, $r = 6, Re = 392$} & (f) {\small PPMD, $r = 6, Re = 104$} \\

\begin{minipage}{0.45\linewidth}
\includegraphics[width=\linewidth]{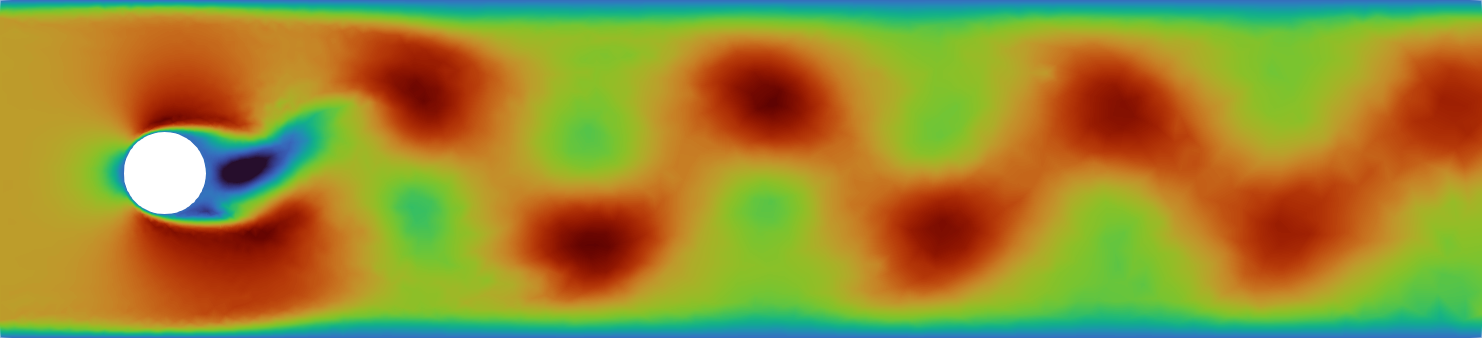}
\end{minipage} &
\begin{minipage}{0.45\linewidth}
\includegraphics[width=\linewidth]{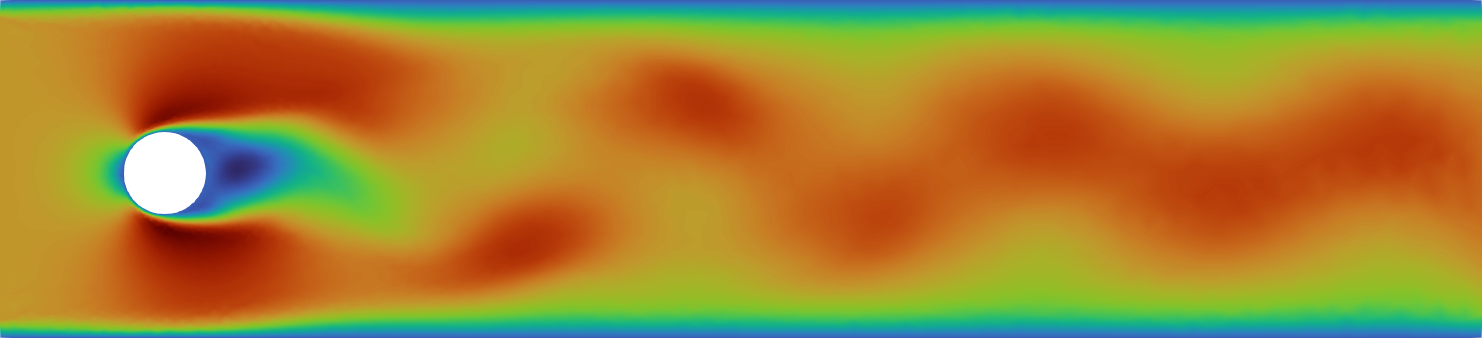}
\end{minipage} \\
(g) {\small POD+GPR, $r = 12, Re = 392$} & (h) {\small POD+GPR, $r = 12, Re = 104$} \\

\begin{minipage}{0.45\linewidth}
\includegraphics[width=\linewidth]{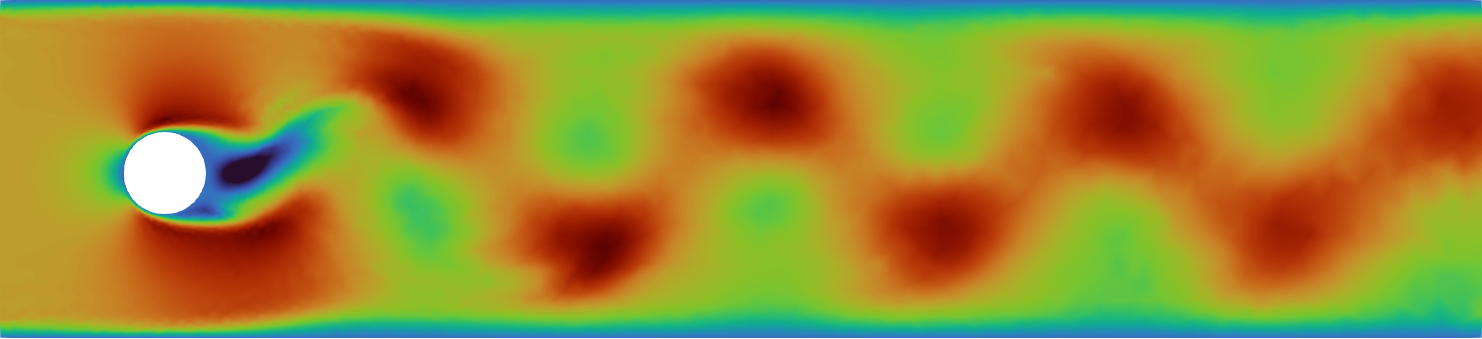}
\end{minipage} &
\begin{minipage}{0.45\linewidth}
\includegraphics[width=\linewidth]{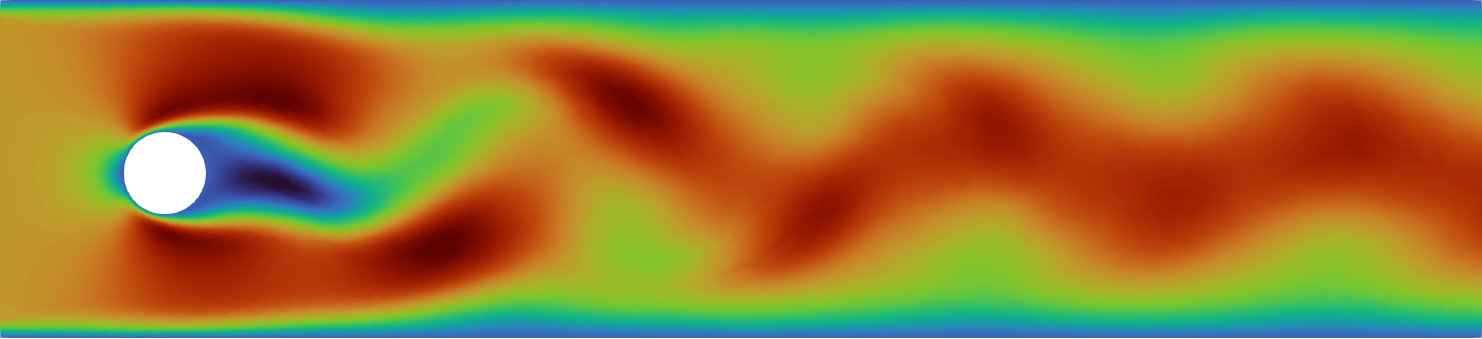}
\end{minipage} \\
(i) {\small PPMD, $r = 12, Re = 392$} & (j) {\small PPMD, $r = 12, Re = 104$} \\

\begin{minipage}{0.45\linewidth} 
\includegraphics[width = \linewidth,angle=0,clip=true]{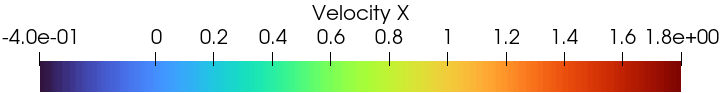}
\end{minipage}
&
\begin{minipage}{0.45\linewidth} 
\includegraphics[width = \linewidth,angle=0,clip=true]{scale3.png}
\end{minipage}\\

\end{tabular}
\caption{Flow past a cylinder: velocity solutions at $Re=392$ and $Re=104$ at $t=15s$ from the high-fidelity full model, POD+GPR, and PPMD using 6 and 12 basis functions.}
\label{fg:flowpredict}
\end{figure}

\begin{figure}[tbhp]
\centering
\begin{tabular}{cc}
\begin{minipage}{0.45\linewidth}
\includegraphics[width=\linewidth]{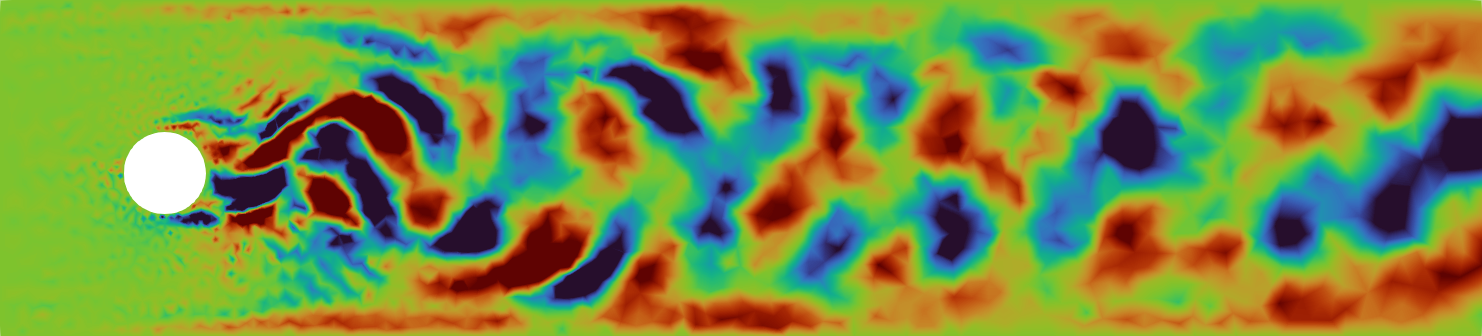}
\end{minipage} &
\begin{minipage}{0.45\linewidth}
\includegraphics[width=\linewidth]{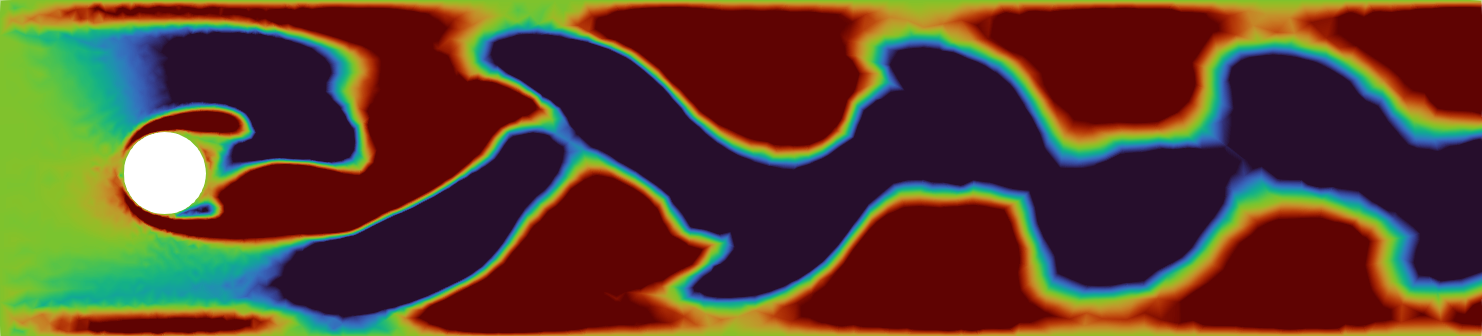}
\end{minipage} \\
(a) {\small POD+GPR, $r = 6, Re = 392$} & (b) {\small POD+GPR, $r = 6, Re =104$} \\

\begin{minipage}{0.45\linewidth}
\includegraphics[width=\linewidth]{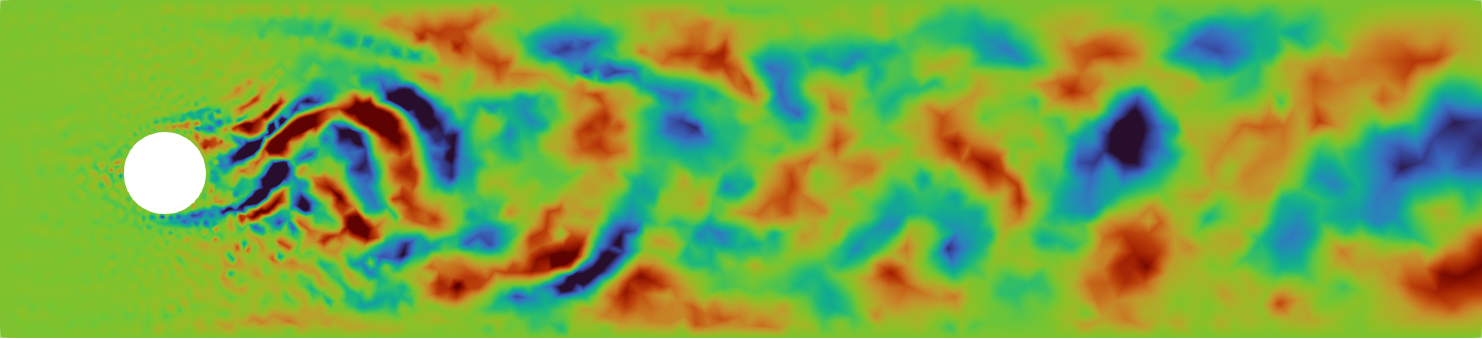}
\end{minipage} &
\begin{minipage}{0.45\linewidth}
\includegraphics[width=\linewidth]{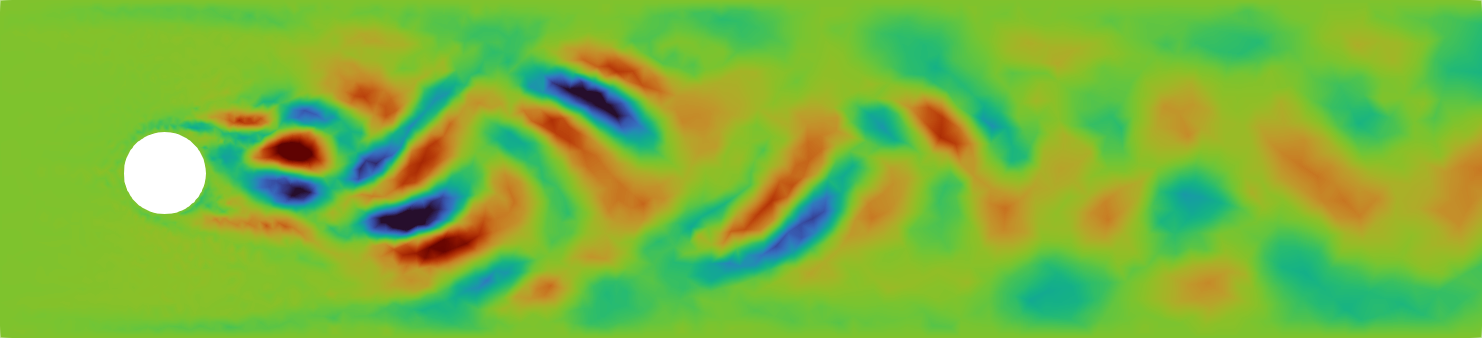}
\end{minipage} \\
(c) {\small PPMD, $r = 6, Re = 392$} & (d) {\small PPMD, $r = 6, Re = 104$} \\

\begin{minipage}{0.45\linewidth}
\includegraphics[width=\linewidth]{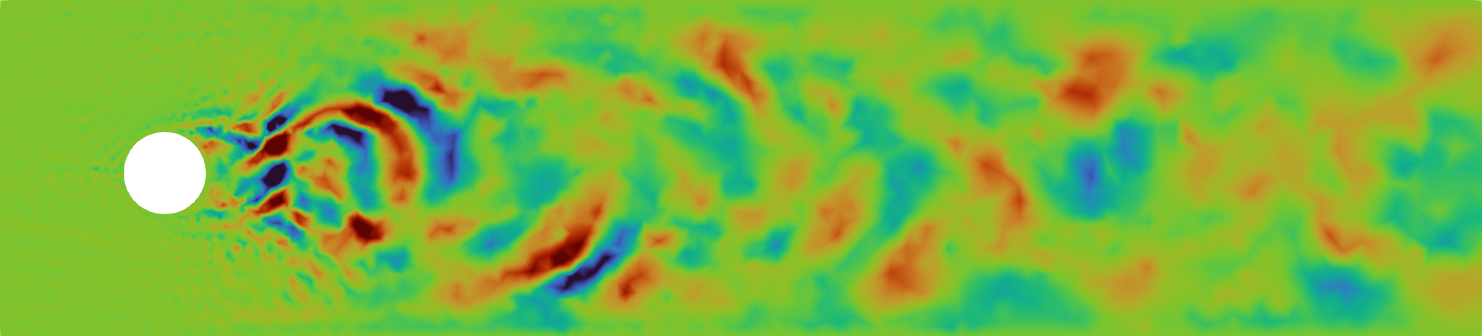}
\end{minipage} &
\begin{minipage}{0.45\linewidth}
\includegraphics[width=\linewidth]{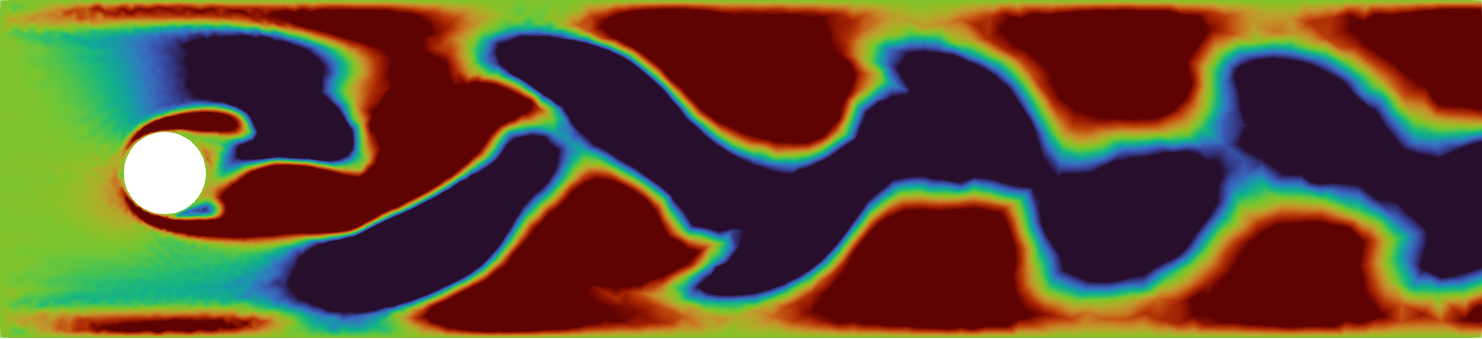}
\end{minipage} \\
(e) {\small POD+GPR, $r = 12, Re = 392$} & (f) {\small POD+GPR, $r = 12, Re = 104$} \\

\begin{minipage}{0.45\linewidth}
\includegraphics[width=\linewidth]{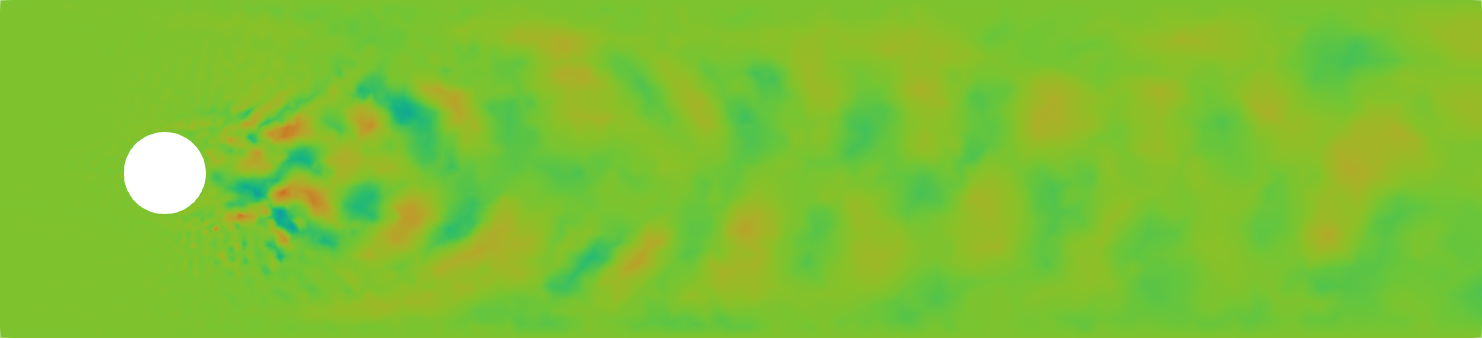}
\end{minipage} &
\begin{minipage}{0.45\linewidth}
\includegraphics[width=\linewidth]{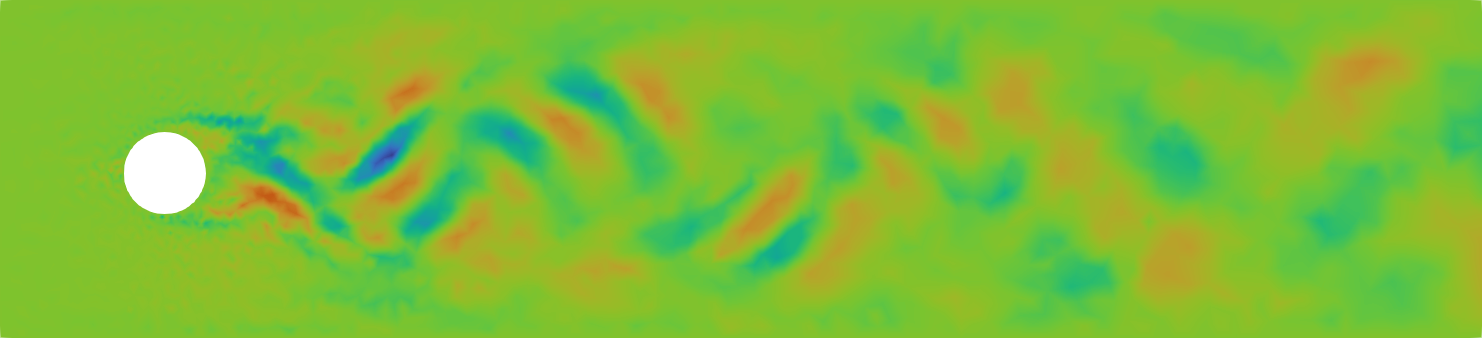}
\end{minipage} \\
(g) {\small PPMD, $r = 12, Re = 392$} & (h) {\small PPMD, $r = 12, Re = 104$} \\

\begin{minipage}{0.45\linewidth} 
\includegraphics[width = \linewidth,angle=0,clip=true]{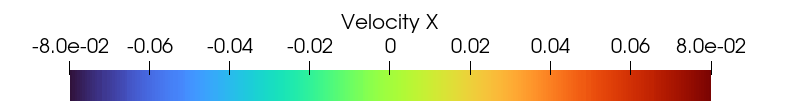}
\end{minipage}
&
\begin{minipage}{0.45\linewidth} 
\includegraphics[width = \linewidth,angle=0,clip=true]{scale4.png}
\end{minipage}\\

\end{tabular}
\caption{Flow past a cylinder test case: prediction errors at $Re=392$ and $Re=104$ at $t=15s$ from the high-fidelity full model for POD+GPR and PPMD using 6 and 12 basis functions.}
\label{fg:flowerror2}
\end{figure}

\Cref{fg:flowreconstruct} shows the velocity fields for the flow past a cylinder at \(Re=250\) and \(t=15\) s from the high-fidelity reference, POD+GPR, and PPMD using 6 and 12 basis functions. The corresponding reconstruction errors are given in \cref{fg:flowerror1}. Visually, POD+GPR with only 6 modes fails to resolve fine-scale vortex structures and wake dynamics; increasing the basis to 12 partially mitigates these deficiencies, but some nonlinear features remain smeared (see \cref{fg:flowreconstruct}(d)). By contrast, PPMD with 6 modes already captures the dominant wake topology and produces substantially lower reconstruction error than POD+GPR, even when the latter uses 12 modes. This indicates that PPMD builds more expressive, parameter-aware latent representations and therefore attains higher reconstruction fidelity with equal or fewer modes.

\begin{table}[htbp]
\centering
\caption{Flow past a cylinder case: MSE and relative error of velocity in low-rank reconstruction and parameter prediction}
\vspace{0.5em}
\renewcommand{\arraystretch}{0.7}
\setlength{\tabcolsep}{8pt}
\footnotesize

\begin{tabular}{llcccc}
\toprule
\textbf{Type} & \textbf{Parameter} & \textbf{Rank} & \textbf{Method} & \textbf{MSE} & \textbf{REL}\\
\midrule
\multirow{4}{*}{Reconstruction} & Re = 250 & 6 & POD+GPR & $1.25 \times 10^{-3}$ & $3.38 \times 10^{-2}$\\
& Re = 250 & 6 & PPMD & $1.36 \times 10^{-4}$ & $1.11 \times 10^{-2}$\\
& Re = 250 & 12 & POD+GPR & $2.36 \times 10^{-4}$ & $1.47 \times 10^{-2}$\\
& Re = 250 & 12 & PPMD & $4.93 \times 10^{-6}$ & $2.12 \times 10^{-3}$\\
\midrule
\multirow{4}{*}{Prediction} & Re = 392 & 6 & POD+GPR & $2.03 \times 10^{-3}$ & $4.32 \times 10^{-2}$\\
& Re = 392 & 6 & PPMD & $8.35 \times 10^{-4}$ & $2.76 \times 10^{-2}$\\
& Re = 392 & 12 & POD+GPR & $4.21 \times 10^{-4}$ & $1.97 \times 10^{-2}$\\
& Re = 392 & 12 & PPMD & $4.61 \times 10^{-5}$ & $6.51 \times 10^{-3}$\\
\midrule
\multirow{4}{*}{Prediction} & Re = 104 & 6 & POD+GPR & $2.71 \times 10^{-2}$ & $1.57 \times 10^{-1}$\\
& Re = 104 & 6 & PPMD & $3.23 \times 10^{-4}$ & $1.71 \times 10^{-2}$\\
& Re = 104 & 12 & POD+GPR & $2.69 \times 10^{-2}$ & $1.57 \times 10^{-1}$\\
& Re = 104 & 12 & PPMD & $9.55 \times 10^{-5}$ & $9.32 \times 10^{-3}$\\
\bottomrule
\end{tabular}
\label{tab:cylinder_velocity_mse}
\end{table}

To evaluate predictive (parametric) performance, the methods are compared at two different Reynolds numbers at \(t=15\) s: \(Re=392\) (interpolation, inside the training range) and \(Re=104\) (extrapolation, outside the training range). \Cref{fg:flowpredict} and \cref{fg:flowerror2} compare the predictive performance of POD+GPR and PPMD using 6 and 12 basis functions. As shown in \cref{fg:flowpredict}, PPMD consistently produces more accurate predictions than POD+GPR for both parameter values using the same number of basis functions. The prediction errors are shown in \cref{fg:flowerror2}. The results reveal a critical difference in the behavior of the two methods. For the in-range parameter (Re=392), increasing the number of basis functions from 6 to 12 enhances the accuracy of both POD+GPR and PPMD, although POD+GPR remains less accurate than PPMD. However, for the out-of-range parameter (Re=104), the performance of POD+GPR shows limited sensitivity to the number of basis functions; increasing the basis count only marginally reduces errors, which remain substantially larger than those of PPMD. In contrast, PPMD maintains high prediction accuracy for both in-range and out-of-range parameters as the number of basis functions increases, demonstrating its superior generalization and robustness across the parameter space.

Quantitative error metrics from Table \ref{tab:cylinder_velocity_mse} further substantiate these observations. For flow reconstruction at Re=250, PPMD with 6 basis functions achieves errors an order of magnitude lower than POD+GPR with the same rank, and even surpasses POD+GPR with 12 basis functions. At Re=392 (in-range), PPMD yields approximately 41\% of the POD+GPR error with 6 basis functions, improving to an order of magnitude lower error with 12 basis functions. The disparity is most pronounced for the out-of-range parameter Re=104, where POD+GPR exhibits large errors (MSE $\sim 2.70 \times 10^{-2}$) showing minimal improvement with increased basis functions, while PPMD maintains errors two orders of magnitude lower.

\subsection{Case 2: Backward-facing step flow}
\label{sec:bfstep_ppmd}
In this experiment, backward-facing step flow is used to demonstrate the capabilities of the PPMD. The kinematic viscosities are varying parameters in this example. 

The computational domain is defined by a backward-facing step geometry using an unstructured triangular mesh with \(N_{\text{nodes}}=5928\) nodes. The flow is governed by the incompressible Navier--Stokes equations with constant density \(\rho = 1\). The kinematic viscosity \(\mu\) is treated as a varying parameter, ranging from \(\mu_{\min}=2.0\times10^{-5}\) to \(\mu_{\max}=2.0\times10^{-3}\) with a uniform increment of \(2.0\times10^{-5}\), yielding 100 distinct parameter values. The corresponding Reynolds numbers range from approximately \(\mathrm{Re}_{\min}\approx 5.75\times10^{2}\) to \(\mathrm{Re}_{\max}\approx 5.75\times10^{4}\), where \(\mathrm{Re}=U_{\text{ref}}L_{\text{ref}}/\mu\), with a reference velocity of \(U_{\text{ref}}=2.3\) m/s and a reference length of \(L_{\text{ref}}=0.5\) m (step height).

\begin{figure}[tbhp]
\centering
\begin{tabular}{cc}
\begin{minipage}{0.45\linewidth}
\includegraphics[width=\linewidth]{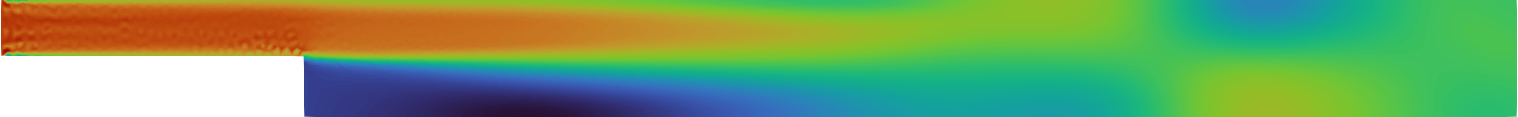}
\end{minipage} &
\begin{minipage}{0.45\linewidth}
\includegraphics[width=\linewidth]{step5e-4full.png}
\end{minipage}\\
(a) {\small Full model, $t = 25s, Re = 2300$} & (b) {\small Full model, $t = 25s, Re = 2300$} \\
\begin{minipage}{0.45\linewidth}
\includegraphics[width=\linewidth]{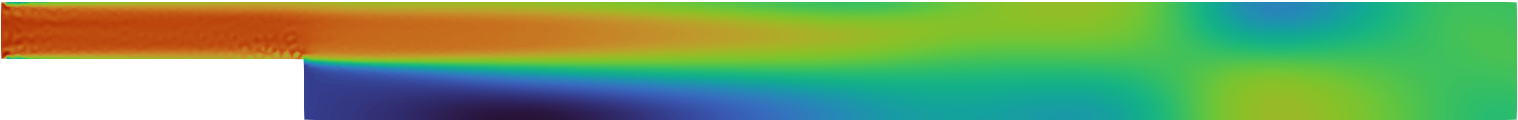}
\end{minipage} &
\begin{minipage}{0.45\linewidth}
\includegraphics[width=\linewidth]{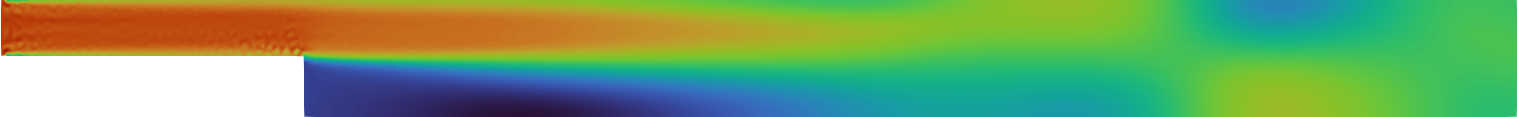}
\end{minipage} \\
(c) {\small POD+GPR, $r = 6$} & (d) {\small POD+GPR, $r = 12$} \\
\begin{minipage}{0.45\linewidth}
\includegraphics[width=\linewidth]{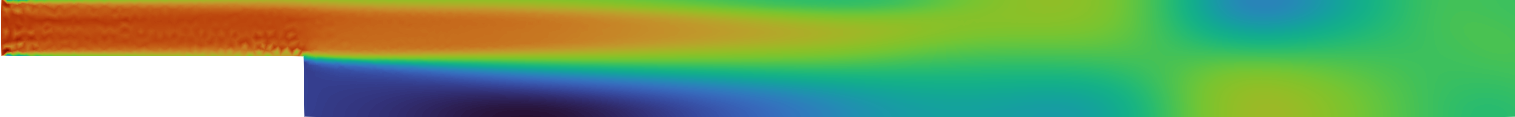}
\end{minipage} &
\begin{minipage}{0.45\linewidth}
\includegraphics[width=\linewidth]{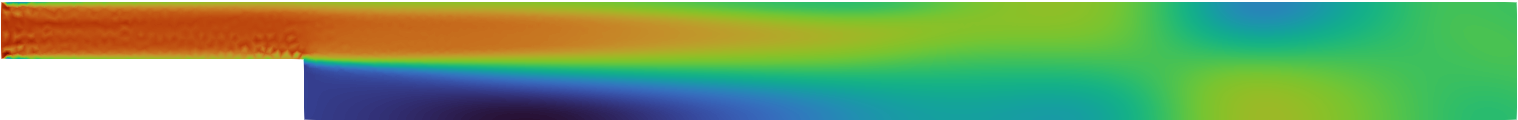}
\end{minipage} \\
(e) {\small PPMD, $r = 6$} & (f) {\small PPMD, $r = 12$} \\
\begin{minipage}{0.45\linewidth} 
\includegraphics[width = \linewidth,angle=0,clip=true]{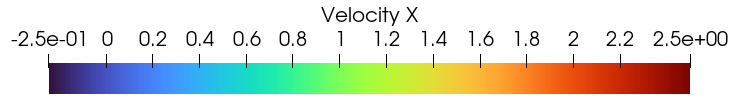}
\end{minipage}
&
\begin{minipage}{0.45\linewidth} 
\includegraphics[width = \linewidth,angle=0,clip=true]{scale6.png}
\end{minipage}\\
\end{tabular}
\caption{Backward-facing step flow at Re=2300: Comparison of reconstruction velocity solutions at $t=25s$ from the high-fidelity full model, POD+GPR, and PPMD using 6 and 12 basis functions.}
\label{fg:stepreconstruct}
\end{figure}

\begin{figure}[tbhp]
\centering
\begin{tabular}{cc}
\begin{minipage}{0.45\linewidth}
\includegraphics[width=\linewidth]{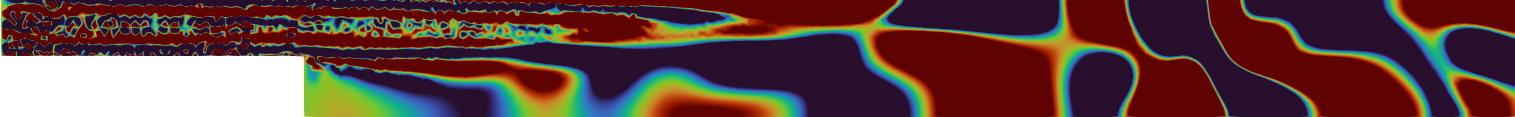}
\end{minipage} &
\begin{minipage}{0.45\linewidth}
\includegraphics[width=\linewidth]{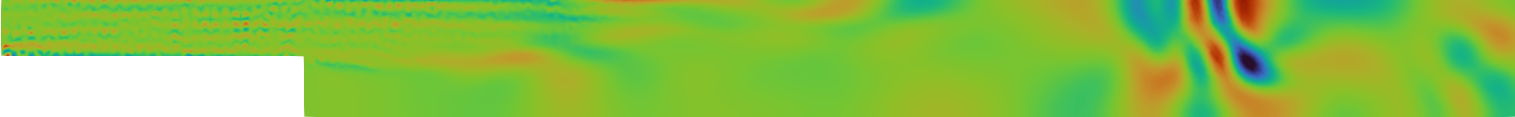}
\end{minipage} \\
(a) {\small POD+GPR, $r = 6$} & (b) {\small POD+GPR, $r = 12$} \\

\begin{minipage}{0.45\linewidth}
\includegraphics[width=\linewidth]{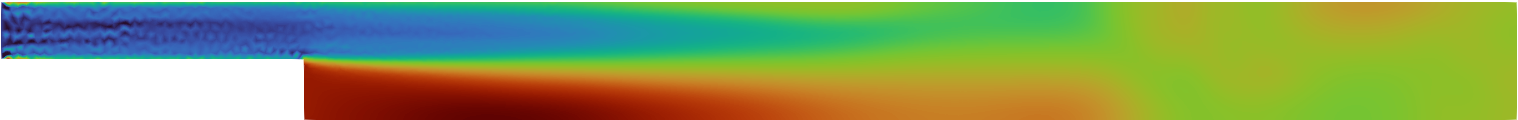}
\end{minipage} &
\begin{minipage}{0.45\linewidth}
\includegraphics[width=\linewidth]{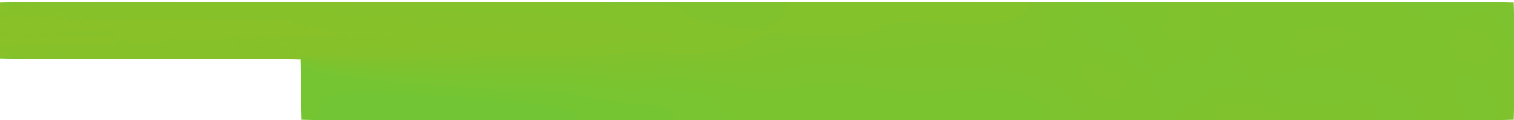}
\end{minipage} \\
(c) {\small PPMD, $r = 6$} & (d) {\small PPMD, $r = 12$} \\

\begin{minipage}{0.45\linewidth} 
\includegraphics[width = \linewidth,angle=0,clip=true]{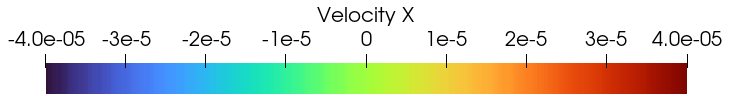}
\end{minipage}
&
\begin{minipage}{0.45\linewidth} 
\includegraphics[width = \linewidth,angle=0,clip=true]{scale5.png}
\end{minipage}\\

\end{tabular}
\caption{Backward-facing step: reconstruction errors at Re=2300 and $t=25s$ for POD+GPR and PPMD using 6 and 12 basis functions.}
\label{fg:steperror}
\end{figure}

Boundary conditions follow the standard backward-facing step configuration: a prescribed horizontal inflow velocity with a maximum \(U_{\text{ref}}=2.3\) m/s, no-slip conditions on all solid walls, and a zero-normal gradient with reference pressure at the outflow. Initial conditions for velocity and pressure are set to zero.

For each parameter value, the system is evolved over the time interval \(t\in[0,25]\) seconds with a fixed time step \(\Delta t=1.0\) s, resulting in 25 snapshots per simulation. The full dataset thus comprises 100 parameter samples, each with 26 snapshots. For training the PPMD model, the first 90 parameter values and their associated snapshots are used. The remaining 10 parameter values are reserved for out-of-sample prediction and error assessment. Additionally, a cubic smoothing spline is fitted to the 100 parameter-snapshot sets to obtain a continuous representation of the solution over the entire parameter interval, enabling reconstruction and prediction at arbitrary parameter values within the range.

\begin{figure}[tbhp]
\centering
\begin{tabular}{cc}
\begin{minipage}{0.45\linewidth}
\includegraphics[width=\linewidth]{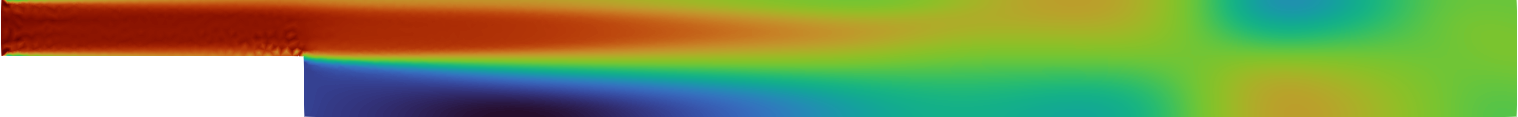}
\end{minipage} &
\begin{minipage}{0.45\linewidth}
\includegraphics[width=\linewidth]{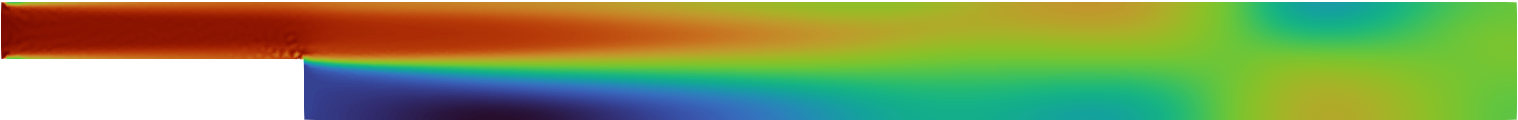}
\end{minipage} \\
(a) {\small Full model, $t = 25s, Re = 1139$} & (b) {\small Full model, $t = 25s, Re = 578$} \\

\begin{minipage}{0.45\linewidth}
\includegraphics[width=\linewidth]{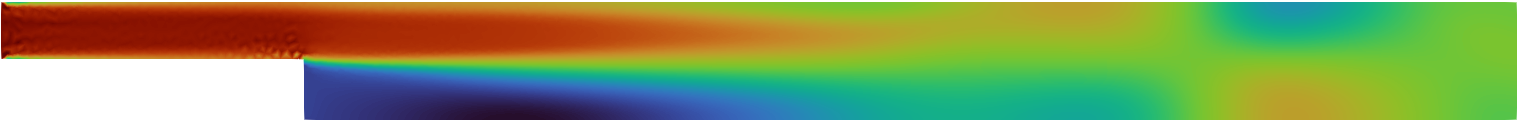}
\end{minipage} &
\begin{minipage}{0.45\linewidth}
\includegraphics[width=\linewidth]{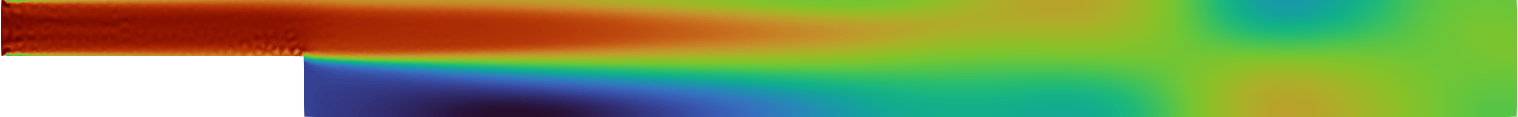}
\end{minipage} \\
(c) {\small POD+GPR, $r = 6, Re = 1139$} & (d) {\small POD+GPR, $r = 6, Re = 578$} \\

\begin{minipage}{0.45\linewidth}
\includegraphics[width=\linewidth]{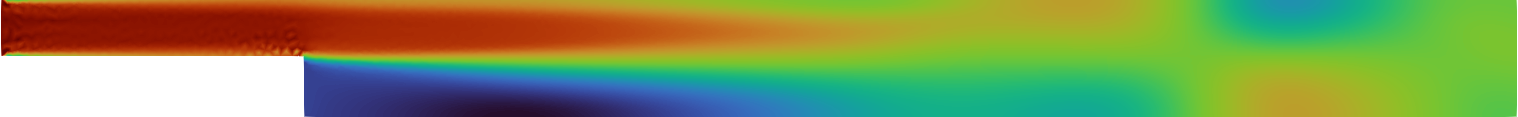}
\end{minipage} &
\begin{minipage}{0.45\linewidth}
\includegraphics[width=\linewidth]{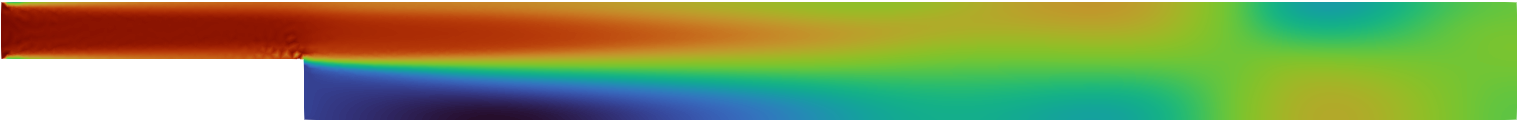}
\end{minipage} \\
(e) {\small PPMD, $r = 6, Re = 1139$} & (f) {\small PPMD, $r = 6, Re = 578$} \\

\begin{minipage}{0.45\linewidth}
\includegraphics[width=\linewidth]{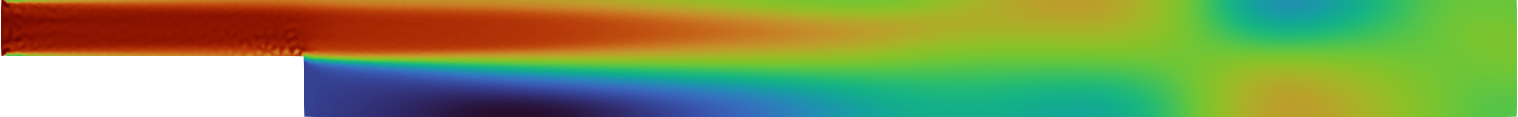}
\end{minipage} &
\begin{minipage}{0.45\linewidth}
\includegraphics[width=\linewidth]{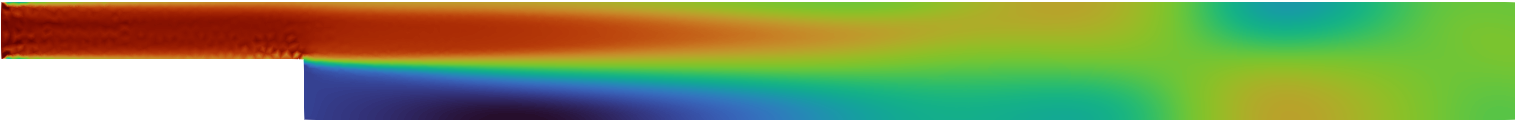}
\end{minipage} \\
(g) {\small POD+GPR, $r = 12, Re = 1139$} & (h) {\small POD+GPR, $r = 12, Re = 578$} \\

\begin{minipage}{0.45\linewidth}
\includegraphics[width=\linewidth]{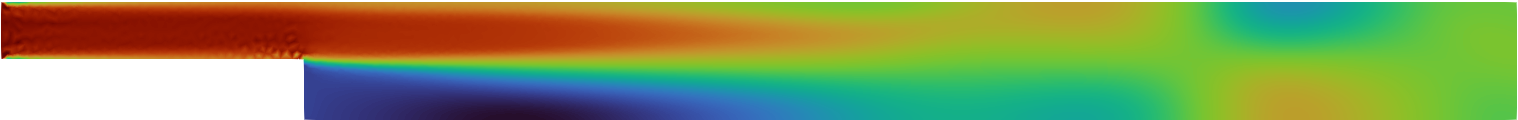}
\end{minipage} &
\begin{minipage}{0.45\linewidth}
\includegraphics[width=\linewidth]{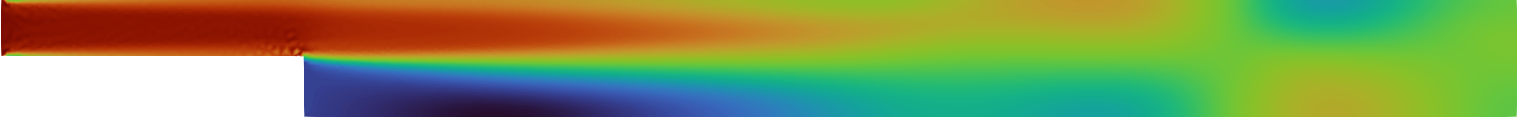}
\end{minipage} \\
(i) {\small PPMD, $r = 12, Re = 1139$} & (j) {\small PPMD, $r = 12, Re = 578$} \\

\begin{minipage}{0.45\linewidth} 
\includegraphics[width = \linewidth,angle=0,clip=true]{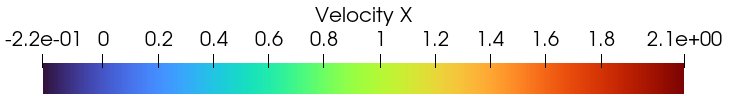}
\end{minipage}
&
\begin{minipage}{0.45\linewidth} 
\includegraphics[width = \linewidth,angle=0,clip=true]{scale8.png}
\end{minipage}\\

\end{tabular}
\caption{Backward-facing step: velocity solutions at $Re=1139$ and $Re=578$ at $t=25s$ from the high-fidelity full model, POD+GPR, and PPMD using 6 and 12 basis functions.}
\label{fg:steppredict}
\end{figure}

\begin{figure}[tbhp]
\centering
\begin{tabular}{cc}
\begin{minipage}{0.45\linewidth}
\includegraphics[width=\linewidth]{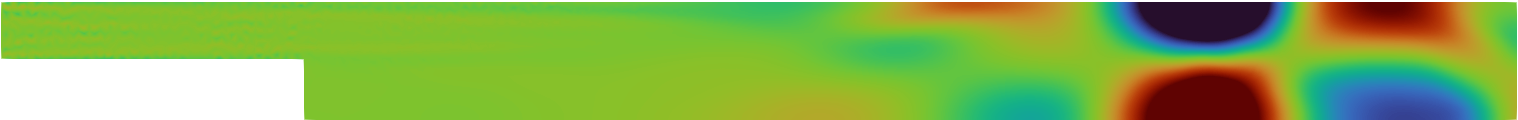}
\end{minipage} &
\begin{minipage}{0.45\linewidth}
\includegraphics[width=\linewidth]{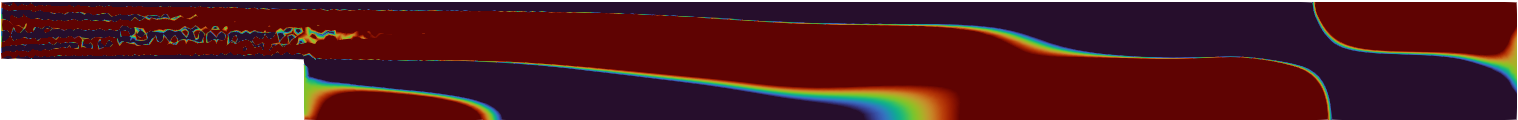}
\end{minipage} \\
(a) {\small POD+GPR, $r = 6, Re = 1139$} & (b) {\small POD+GPR, $r = 6, Re =578$} \\

\begin{minipage}{0.45\linewidth}
\includegraphics[width=\linewidth]{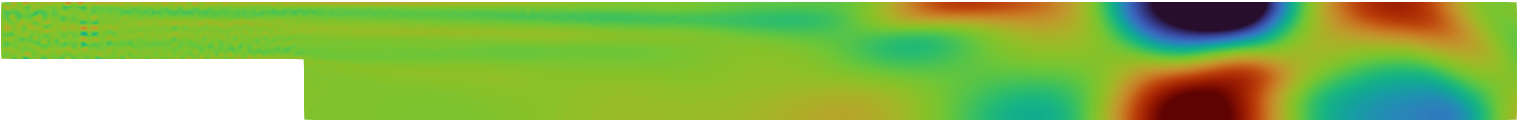}
\end{minipage} &
\begin{minipage}{0.45\linewidth}
\includegraphics[width=\linewidth]{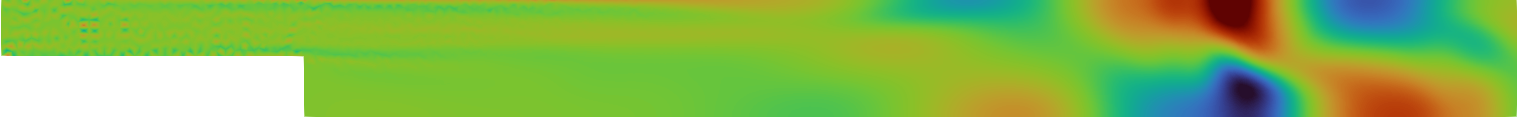}
\end{minipage} \\
(c) {\small PPMD, $r = 6, Re = 1139$} & (d) {\small PPMD, $r = 6, Re = 578$} \\

\begin{minipage}{0.45\linewidth}
\includegraphics[width=\linewidth]{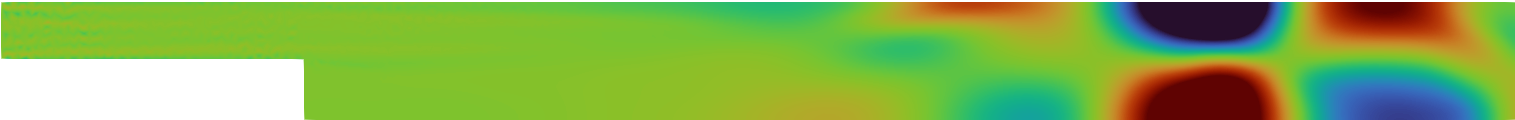}
\end{minipage} &
\begin{minipage}{0.45\linewidth}
\includegraphics[width=\linewidth]{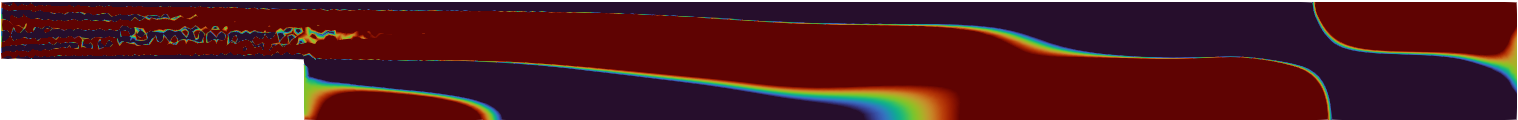}
\end{minipage} \\
(e) {\small POD+GPR, $r = 12, Re = 1139$} & (f) {\small POD+GPR, $r = 12, Re = 578$} \\

\begin{minipage}{0.45\linewidth}
\includegraphics[width=\linewidth]{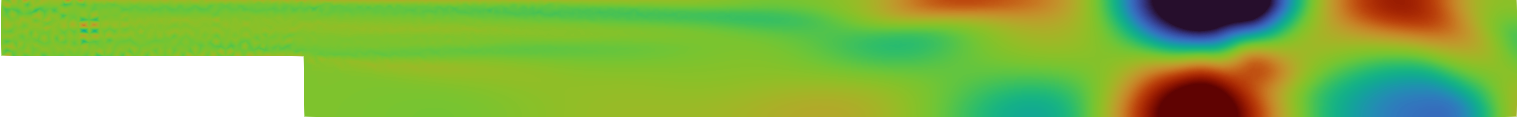}
\end{minipage} &
\begin{minipage}{0.45\linewidth}
\includegraphics[width=\linewidth]{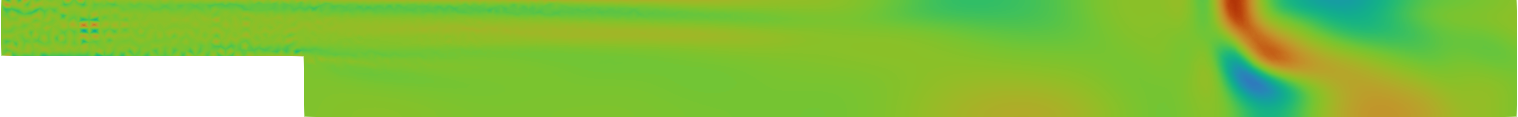}
\end{minipage} \\
(g) {\small PPMD, $r = 12, Re = 1139$} & (h) {\small PPMD, $r = 12, Re = 578$} \\

\begin{minipage}{0.45\linewidth} 
\includegraphics[width = \linewidth,angle=0,clip=true]{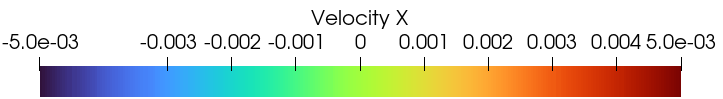}
\end{minipage}
&
\begin{minipage}{0.45\linewidth} 
\includegraphics[width = \linewidth,angle=0,clip=true]{scale7.png}
\end{minipage}\\

\end{tabular}
\caption{Backward-facing step test case: prediction errors at $Re=1139$ and $Re=578$ at $t=25s$ from the high-fidelity full model for POD+GPR and PPMD using 6 and 12 basis functions.}
\label{fg:steperror2}
\end{figure}

\Cref{fg:stepreconstruct} shows the velocity fields for the 2D backward-facing step at \(t=25\) s and \(Re=2300\) from the high-fidelity model, POD+GPR, and PPMD using 6 and 12 bases. The reconstruction errors are shown in \cref{fg:steperror}. POD+GPR with 6 modes under-resolves the separated shear layer and downstream recirculation, and although 12 modes reduce these deficiencies, some nonlinear features remain smeared. PPMD with 6 modes, by contrast, captures the main wake topology and near-step dynamics more faithfully and attains a lower reconstruction error than POD+GPR (even with 12 modes).

To assess parametric extrapolation and interpolation performance, two prediction cases (\(t=25\) s) \(Re=1139\) (in-range) and \(Re=578\) (out-of-range) are considered. The results are shown in \Cref{fg:steppredict} and \cref{fg:steperror2}. In the in-range case, both methods perform comparably with 6 modes. The pointwise maximum error is already small (\(\mathcal{O}(5\times10^{-3})\)); thus, increasing to 12 modes yields only marginal improvement. For the out-of-range case, PPMD markedly outperforms POD+GPR; POD+GPR shows little sensitivity to added modes and fails to recover key wake features, whereas PPMD improves noticeably as the modal dimension increases. Overall, PPMD generalizes more robustly across the parameter space and benefits more from increased modal resolution.

\section{Conclusions}  \label{sec:summary}
\emph{Parametric Probabilistic Manifold Decomposition} (PPMD), an extension of the PMD paradigm tailored to parameter-dependent systems, is introduced. The two-stage spirit of PMD (linear reduction + nonlinear residual modeling) is retained, but two key departures are made. First, snapshots are assembled by concatenating full time-resolved trajectories across parameter samples, so that parametric variation is modeled jointly with temporal structure, rather than with time treated as the sole ordering variable. Second, a recursive, adaptive KRR scheme is introduced to augment latent samples and improve local coverage in parameter space; subsequently, discrete latent samples are converted into differentiable, continuous parameter-to-latent maps by weighted smoothing splines. Continuous-parameter surrogates with improved generalization and controllable error sources are thereby yielded.

Numerical experiments on two canonical CFD benchmarks (flow past a cylinder and the backward-facing step) demonstrate that PPMD attains competitive reconstruction accuracy while producing continuous, interpretable parametric surrogates with lower training complexity and clearer modular error control.  Unlike purely black-box deep models, PPMD provides explicit linear and nonlinear components that aid in diagnosis and calibration; in addition, its recursive KRR augmentation and spline extension improve pointwise predictions at unseen parameter values.  

The principal limitations remain the representational capacity of the chosen kernel RKHS and the computational cost of kernel operations; practical remedies include richer kernels, low-rank local KRR approximations, and adaptive sampling guided by uncertainty estimates.  Future work will target physics-aware lifting, uncertainty-driven adaptive sampling, and large-scale benchmarks (multi-phase flows, urban ocean models) to further evaluate scalability and robustness.

\section{ Acknowledgments}

The authors acknowledge the support of the Top Discipline Plan of Shanghai Universities-Class I and Shanghai Municipal Science and Technology Major Project (No. 2021SHZDZX0100), National Key $R\&D$ Program of China(NO.2022YFE0208000, 2024YFC2816400, and 2024YFC2816401), Shanghai Engineering Research Center(No. 19DZ2255100) and the Shanghai Institute of Intelligent Science and Technology, Tongji University.	

\bibliographystyle{siamplain}
\bibliography{references}
\end{document}